\documentclass[11pt, reqno]{amsart}

\usepackage{amssymb}
\usepackage{eucal}
\usepackage{amsmath}
\usepackage{amscd}
\usepackage[dvips]{color}
\usepackage{multicol}
\usepackage[all]{xy}
\usepackage{graphicx}
\usepackage{color}
\usepackage{colordvi}
\usepackage{xspace}

\begin{document}
\newcommand {\emptycomment}[1]{} 

\baselineskip=14pt
\newcommand{\nc}{\newcommand}
\newcommand{\delete}[1]{}
\nc{\mfootnote}[1]{\footnote{#1}}
\nc{\todo}[1]{\tred{To do:} #1}

\nc{\mlabel}[1]{\label{#1}}  
\nc{\mcite}[1]{\cite{#1}}  
\nc{\mref}[1]{\ref{#1}}  
\nc{\mbibitem}[1]{\bibitem{#1}} 

\delete{
\nc{\mlabel}[1]{\label{#1}  
{\hfill \hspace{1cm}{\bf{{\ }\hfill(#1)}}}}
\nc{\mcite}[1]{\cite{#1}{{\bf{{\ }(#1)}}}}  
\nc{\mref}[1]{\ref{#1}{{\bf{{\ }(#1)}}}}  
\nc{\mbibitem}[1]{\bibitem[\bf #1]{#1}} 
}

\newtheorem{thm}{Theorem}[section]
\newtheorem{lem}[thm]{Lemma}
\newtheorem{cor}[thm]{Corollary}
\newtheorem{pro}[thm]{Proposition}
\newtheorem{ex}[thm]{Example}
\newtheorem{rmk}[thm]{Remark}
\newtheorem{defi}[thm]{Definition}
\newtheorem{pdef}[thm]{Proposition-Definition}
\newtheorem{condition}[thm]{Condition}
\newtheorem{notation}[thm]{Notation}

\renewcommand{\labelenumi}{{\rm(\alph{enumi})}}
\renewcommand{\theenumi}{\alph{enumi}}

\nc{\tred}[1]{\textcolor{red}{#1}}
\nc{\tblue}[1]{\textcolor{blue}{#1}}
\nc{\tgreen}[1]{\textcolor{green}{#1}}
\nc{\tpurple}[1]{\textcolor{purple}{#1}}
\nc{\btred}[1]{\textcolor{red}{\bf #1}}
\nc{\btblue}[1]{\textcolor{blue}{\bf #1}}
\nc{\btgreen}[1]{\textcolor{green}{\bf #1}}
\nc{\btpurple}[1]{\textcolor{purple}{\bf #1}}

\nc{\py}[1]{\textcolor{red}{Panyang:#1}}
\nc{\yy}[1]{\textcolor{blue}{Yanyong: #1}}
\nc{\lit}[2]{\textcolor{blue}{#1}{}} 
\nc{\yh}[1]{\textcolor{green}{Yunhe: #1}}


\nc{\twovec}[2]{\left(\begin{array}{c} #1 \\ #2\end{array} \right )}
\nc{\threevec}[3]{\left(\begin{array}{c} #1 \\ #2 \\ #3 \end{array}\right )}
\nc{\twomatrix}[4]{\left(\begin{array}{cc} #1 & #2\\ #3 & #4 \end{array} \right)}
\nc{\threematrix}[9]{{\left(\begin{matrix} #1 & #2 & #3\\ #4 & #5 & #6 \\ #7 & #8 & #9 \end{matrix} \right)}}
\nc{\twodet}[4]{\left|\begin{array}{cc} #1 & #2\\ #3 & #4 \end{array} \right|}

\nc{\rk}{\mathrm{r}}
\newcommand{\g}{\mathfrak g}
\newcommand{\h}{\mathfrak h}
\newcommand{\pf}{\noindent{$Proof$.}\ }
\newcommand{\frkg}{\mathfrak g}
\newcommand{\frkh}{\mathfrak h}
\newcommand{\Id}{\rm{Id}}
\newcommand{\gl}{\mathfrak {gl}}
\newcommand{\ad}{\mathrm{ad}}
\newcommand{\add}{\frka\frkd}
\newcommand{\frka}{\mathfrak a}
\newcommand{\frkb}{\mathfrak b}
\newcommand{\frkc}{\mathfrak c}
\newcommand{\frkd}{\mathfrak d}
\newcommand {\comment}[1]{{\marginpar{*}\scriptsize\textbf{Comments:} #1}}

\nc{\gensp}{V} 
\nc{\relsp}{\Lambda} 
\nc{\leafsp}{X}    
\nc{\treesp}{\overline{\calt}} 

\nc{\vin}{{\mathrm Vin}}    
\nc{\lin}{{\mathrm Lin}}    

\nc{\gop}{{\,\omega\,}}     
\nc{\gopb}{{\,\nu\,}}
\nc{\svec}[2]{{\tiny\left(\begin{matrix}#1\\
#2\end{matrix}\right)\,}}  
\nc{\ssvec}[2]{{\tiny\left(\begin{matrix}#1\\
#2\end{matrix}\right)\,}} 

\nc{\typeI}{local cocycle $3$-Lie bialgebra\xspace}
\nc{\typeIs}{local cocycle $3$-Lie bialgebras\xspace}
\nc{\typeII}{double construction $3$-Lie bialgebra\xspace}
\nc{\typeIIs}{double construction $3$-Lie bialgebras\xspace}

\nc{\bia}{{$\mathcal{P}$-bimodule ${\bf k}$-algebra}\xspace}
\nc{\bias}{{$\mathcal{P}$-bimodule ${\bf k}$-algebras}\xspace}

\nc{\rmi}{{\mathrm{I}}}
\nc{\rmii}{{\mathrm{II}}}
\nc{\rmiii}{{\mathrm{III}}}
\nc{\pr}{{\mathrm{pr}}}
\newcommand{\huaA}{\mathcal{A}}
\nc{\pll}{\beta}
\nc{\plc}{\epsilon}
\nc{\ass}{{\mathit{Ass}}}
\nc{\lie}{{\mathit{Lie}}}
\nc{\comm}{{\mathit{Comm}}}
\nc{\dend}{{\mathit{Dend}}}
\nc{\zinb}{{\mathit{Zinb}}}
\nc{\tdend}{{\mathit{TDend}}}
\nc{\prelie}{{\mathit{preLie}}}
\nc{\postlie}{{\mathit{PostLie}}}
\nc{\quado}{{\mathit{Quad}}}
\nc{\octo}{{\mathit{Octo}}}
\nc{\ldend}{{\mathit{ldend}}}
\nc{\lquad}{{\mathit{LQuad}}}

 \nc{\adec}{\check{;}} \nc{\aop}{\alpha}
\nc{\dftimes}{\widetilde{\otimes}} \nc{\dfl}{\succ} \nc{\dfr}{\prec}
\nc{\dfc}{\circ} \nc{\dfb}{\bullet} \nc{\dft}{\star}
\nc{\dfcf}{{\mathbf k}} \nc{\apr}{\ast} \nc{\spr}{\cdot}
\nc{\twopr}{\circ} \nc{\tspr}{\star} \nc{\sempr}{\ast}
\nc{\disp}[1]{\displaystyle{#1}}
\nc{\bin}[2]{ (_{\stackrel{\scs{#1}}{\scs{#2}}})}  
\nc{\binc}[2]{ \left (\!\! \begin{array}{c} \scs{#1}\\
    \scs{#2} \end{array}\!\! \right )}  
\nc{\bincc}[2]{  \left ( {\scs{#1} \atop
    \vspace{-.5cm}\scs{#2}} \right )}  
\nc{\sarray}[2]{\begin{array}{c}#1 \vspace{.1cm}\\ \hline
    \vspace{-.35cm} \\ #2 \end{array}}
\nc{\bs}{\bar{S}} \nc{\dcup}{\stackrel{\bullet}{\cup}}
\nc{\dbigcup}{\stackrel{\bullet}{\bigcup}} \nc{\etree}{\big |}
\nc{\la}{\longrightarrow} \nc{\fe}{\'{e}} \nc{\rar}{\rightarrow}
\nc{\dar}{\downarrow} \nc{\dap}[1]{\downarrow
\rlap{$\scriptstyle{#1}$}} \nc{\uap}[1]{\uparrow
\rlap{$\scriptstyle{#1}$}} \nc{\defeq}{\stackrel{\rm def}{=}}
\nc{\dis}[1]{\displaystyle{#1}} \nc{\dotcup}{\,
\displaystyle{\bigcup^\bullet}\ } \nc{\sdotcup}{\tiny{
\displaystyle{\bigcup^\bullet}\ }} \nc{\hcm}{\ \hat{,}\ }
\nc{\hcirc}{\hat{\circ}} \nc{\hts}{\hat{\shpr}}
\nc{\lts}{\stackrel{\leftarrow}{\shpr}}
\nc{\rts}{\stackrel{\rightarrow}{\shpr}} \nc{\lleft}{[}
\nc{\lright}{]} \nc{\uni}[1]{\tilde{#1}} \nc{\wor}[1]{\check{#1}}
\nc{\free}[1]{\bar{#1}} \nc{\den}[1]{\check{#1}} \nc{\lrpa}{\wr}
\nc{\curlyl}{\left \{ \begin{array}{c} {} \\ {} \end{array}
    \right .  \!\!\!\!\!\!\!}
\nc{\curlyr}{ \!\!\!\!\!\!\!
    \left . \begin{array}{c} {} \\ {} \end{array}
    \right \} }
\nc{\leaf}{\ell}       
\nc{\longmid}{\left | \begin{array}{c} {} \\ {} \end{array}
    \right . \!\!\!\!\!\!\!}
\nc{\ot}{\otimes} \nc{\sot}{{\scriptstyle{\ot}}}
\nc{\otm}{\overline{\ot}}
\nc{\ora}[1]{\stackrel{#1}{\rar}}
\nc{\ola}[1]{\stackrel{#1}{\la}}
\nc{\pltree}{\calt^\pl}
\nc{\epltree}{\calt^{\pl,\NC}}
\nc{\rbpltree}{\calt^r}
\nc{\scs}[1]{\scriptstyle{#1}} \nc{\mrm}[1]{{\rm #1}}
\nc{\dirlim}{\displaystyle{\lim_{\longrightarrow}}\,}
\nc{\invlim}{\displaystyle{\lim_{\longleftarrow}}\,}
\nc{\mvp}{\vspace{0.5cm}} \nc{\svp}{\vspace{2cm}}
\nc{\vp}{\vspace{8cm}} \nc{\proofbegin}{\noindent{\bf Proof: }}
\nc{\proofend}{$\blacksquare$ \vspace{0.5cm}}
\nc{\freerbpl}{{F^{\mathrm RBPL}}}
\nc{\sha}{{\mbox{\cyr X}}}  
\nc{\ncsha}{{\mbox{\cyr X}^{\mathrm NC}}} \nc{\ncshao}{{\mbox{\cyr
X}^{\mathrm NC,\,0}}}
\nc{\shpr}{\diamond}    
\nc{\shprm}{\overline{\diamond}}    
\nc{\shpro}{\diamond^0}    
\nc{\shprr}{\diamond^r}     
\nc{\shpra}{\overline{\diamond}^r}
\nc{\shpru}{\check{\diamond}} \nc{\catpr}{\diamond_l}
\nc{\rcatpr}{\diamond_r} \nc{\lapr}{\diamond_a}
\nc{\sqcupm}{\ot}
\nc{\lepr}{\diamond_e} \nc{\vep}{\varepsilon} \nc{\labs}{\mid\!}
\nc{\rabs}{\!\mid} \nc{\hsha}{\widehat{\sha}}
\nc{\lsha}{\stackrel{\leftarrow}{\sha}}
\nc{\rsha}{\stackrel{\rightarrow}{\sha}} \nc{\lc}{\lfloor}
\nc{\rc}{\rfloor}
\nc{\tpr}{\sqcup}
\nc{\nctpr}{\vee}
\nc{\plpr}{\star}
\nc{\rbplpr}{\bar{\plpr}}
\nc{\sqmon}[1]{\langle #1\rangle}
\nc{\forest}{\calf}
\nc{\altx}{\Lambda_X} \nc{\vecT}{\vec{T}} \nc{\onetree}{\bullet}
\nc{\Ao}{\check{A}}
\nc{\seta}{\underline{\Ao}}
\nc{\deltaa}{\overline{\delta}}
\nc{\trho}{\tilde{\rho}}

\nc{\rpr}{\circ}
\nc{\dpr}{{\tiny\diamond}}
\nc{\rprpm}{{\rpr}}

\nc{\mmbox}[1]{\mbox{\ #1\ }} \nc{\ann}{\mrm{ann}}
\nc{\Aut}{\mrm{Aut}} \nc{\can}{\mrm{can}}
\nc{\twoalg}{{two-sided algebra}\xspace}
\nc{\colim}{\mrm{colim}}
\nc{\Cont}{\mrm{Cont}} \nc{\rchar}{\mrm{char}}
\nc{\cok}{\mrm{coker}} \nc{\dtf}{{R-{\rm tf}}} \nc{\dtor}{{R-{\rm
tor}}}
\renewcommand{\det}{\mrm{det}}
\nc{\depth}{{\mrm d}}
\nc{\Div}{{\mrm Div}} \nc{\End}{\mrm{End}} \nc{\Ext}{\mrm{Ext}}
\nc{\Fil}{\mrm{Fil}} \nc{\Frob}{\mrm{Frob}} \nc{\Gal}{\mrm{Gal}}
\nc{\GL}{\mrm{GL}} \nc{\Hom}{\mrm{Hom}} \nc{\hsr}{\mrm{H}}
\nc{\hpol}{\mrm{HP}} \nc{\id}{\mrm{id}} \nc{\im}{\mrm{im}}
\nc{\incl}{\mrm{incl}} \nc{\length}{\mrm{length}}
\nc{\LR}{\mrm{LR}} \nc{\mchar}{\rm char} \nc{\NC}{\mrm{NC}}
\nc{\mpart}{\mrm{part}} \nc{\pl}{\mrm{PL}}
\nc{\ql}{{\QQ_\ell}} \nc{\qp}{{\QQ_p}}
\nc{\rank}{\mrm{rank}} \nc{\rba}{\rm{RBA }} \nc{\rbas}{\rm{RBAs }}
\nc{\rbpl}{\mrm{RBPL}}
\nc{\rbw}{\rm{RBW }} \nc{\rbws}{\rm{RBWs }} \nc{\rcot}{\mrm{cot}}
\nc{\rest}{\rm{controlled}\xspace}
\nc{\rdef}{\mrm{def}} \nc{\rdiv}{{\rm div}} \nc{\rtf}{{\rm tf}}
\nc{\rtor}{{\rm tor}} \nc{\res}{\mrm{res}} \nc{\SL}{\mrm{SL}}
\nc{\Spec}{\mrm{Spec}} \nc{\tor}{\mrm{tor}} \nc{\Tr}{\mrm{Tr}}
\nc{\mtr}{\mrm{sk}}

\nc{\ab}{\mathbf{Ab}} \nc{\Alg}{\mathbf{Alg}}
\nc{\Algo}{\mathbf{Alg}^0} \nc{\Bax}{\mathbf{Bax}}
\nc{\Baxo}{\mathbf{Bax}^0} \nc{\RB}{\mathbf{RB}}
\nc{\RBo}{\mathbf{RB}^0} \nc{\BRB}{\mathbf{RB}}
\nc{\Dend}{\mathbf{DD}} \nc{\bfk}{{\bf k}} \nc{\bfone}{{\bf 1}}
\nc{\base}[1]{{a_{#1}}} \nc{\detail}{\marginpar{\bf More detail}
    \noindent{\bf Need more detail!}
    \svp}
\nc{\Diff}{\mathbf{Diff}} \nc{\gap}{\marginpar{\bf
Incomplete}\noindent{\bf Incomplete!!}
    \svp}
\nc{\FMod}{\mathbf{FMod}} \nc{\mset}{\mathbf{MSet}}
\nc{\rb}{\mathrm{RB}} \nc{\Int}{\mathbf{Int}}
\nc{\Mon}{\mathbf{Mon}}
\nc{\remarks}{\noindent{\bf Remarks: }}
\nc{\OS}{\mathbf{OS}} 
\nc{\Rep}{\mathbf{Rep}}
\nc{\Rings}{\mathbf{Rings}} \nc{\Sets}{\mathbf{Sets}}
\nc{\DT}{\mathbf{DT}}

\nc{\BA}{{\mathbb A}} \nc{\CC}{{\mathbb C}} \nc{\DD}{{\mathbb D}}
\nc{\EE}{{\mathbb E}} \nc{\FF}{{\mathbb F}} \nc{\GG}{{\mathbb G}}
\nc{\HH}{{\mathbb H}} \nc{\LL}{{\mathbb L}} \nc{\NN}{{\mathbb N}}
\nc{\QQ}{{\mathbb Q}} \nc{\RR}{{\mathbb R}} \nc{\BS}{{\mathbb{S}}} \nc{\TT}{{\mathbb T}}
\nc{\VV}{{\mathbb V}} \nc{\ZZ}{{\mathbb Z}}


\nc{\calao}{{\mathcal A}} \nc{\cala}{{\mathcal A}}
\nc{\calc}{{\mathcal C}} \nc{\cald}{{\mathcal D}}
\nc{\cale}{{\mathcal E}} \nc{\calf}{{\mathcal F}}
\nc{\calfr}{{{\mathcal F}^{\,r}}} \nc{\calfo}{{\mathcal F}^0}
\nc{\calfro}{{\mathcal F}^{\,r,0}} \nc{\oF}{\overline{F}}
\nc{\calg}{{\mathcal G}} \nc{\calh}{{\mathcal H}}
\nc{\cali}{{\mathcal I}} \nc{\calj}{{\mathcal J}}
\nc{\call}{{\mathcal L}} \nc{\calm}{{\mathcal M}}
\nc{\caln}{{\mathcal N}} \nc{\calo}{{\mathcal O}}
\nc{\calp}{{\mathcal P}} \nc{\calq}{{\mathcal Q}} \nc{\calr}{{\mathcal R}}
\nc{\calt}{{\mathcal T}} \nc{\caltr}{{\mathcal T}^{\,r}}
\nc{\calu}{{\mathcal U}} \nc{\calv}{{\mathcal V}}
\nc{\calw}{{\mathcal W}} \nc{\calx}{{\mathcal X}}
\nc{\CA}{\mathcal{A}}

\nc{\fraka}{{\mathfrak a}} \nc{\frakB}{{\mathfrak B}}
\nc{\frakb}{{\mathfrak b}} \nc{\frakd}{{\mathfrak d}}
\nc{\oD}{\overline{D}}
\nc{\frakF}{{\mathfrak F}} \nc{\frakg}{{\mathfrak g}}
\nc{\frakm}{{\mathfrak m}} \nc{\frakM}{{\mathfrak M}}
\nc{\frakMo}{{\mathfrak M}^0} \nc{\frakp}{{\mathfrak p}}
\nc{\frakS}{{\mathfrak S}} \nc{\frakSo}{{\mathfrak S}^0}
\nc{\fraks}{{\mathfrak s}} \nc{\os}{\overline{\fraks}}
\nc{\frakT}{{\mathfrak T}}
\nc{\oT}{\overline{T}}
\nc{\frakX}{{\mathfrak X}} \nc{\frakXo}{{\mathfrak X}^0}
\nc{\frakx}{{\mathbf x}}
\nc{\frakTx}{\frakT}      
\nc{\frakTa}{\frakT^a}        
\nc{\frakTxo}{\frakTx^0}   
\nc{\caltao}{\calt^{a,0}}   
\nc{\ox}{\overline{\frakx}} \nc{\fraky}{{\mathfrak y}}
\nc{\frakz}{{\mathfrak z}} \nc{\oX}{\overline{X}}

\font\cyr=wncyr10

\nc{\redtext}[1]{\textcolor{red}{#1}}

\title{On a class of infinite simple Lie conformal algebras}

\author{Yanyong Hong}
\address{Department of Mathematics, Hangzhou Normal University,
Hangzhou, 311121, P.R.China}
\email{hongyanyong2008@yahoo.com}

\author{Yang Pan}
\address{Department of Mathematics and Physics, Hefei University,
 Hefei, 230601, P.R. China}
\email{ypan@outlook.de}

\author{Haibo Chen}
\address{ School of  Statistics and Mathematics, Shanghai Lixin University of  Accounting and Finance,   Shanghai,201209, P.R. China}
\email{rebel1025@126.com}

\subjclass[2010]{16D70, 16S32, 16S99, 16W20}
\keywords{Lie conformal algebra, Representation, Central extensions, Conformal derivation, Intermediate series module}

\begin{abstract}
In this paper, we study a class of infinite simple Lie conformal algebras associated to a class of generalized Block type Lie algebras. The central extensions, conformal derivations and free intermediate series modules of this class of Lie conformal algebras are determined. Moreover, we also show that these Lie conformal algebras do not have any non-trivial finite conformal modules. Consequently, these Lie conformal algebras cannot be embedded into $gc_N$ for any positive integer $N$.
\end{abstract}

\maketitle

\section{Introduction}
The notion of Lie conformal algebra, introduced by
V.Kac, gives an axiomatic description of the operator product expansion (or rather
its Fourier transform) of chiral fields in conformal field theory. It plays important roles in quantum field theory and vertex operator algebras (see \cite{K1}). Moreover, it has many applications in the theory of infinite-dimensional Lie algebras satisfying the locality property in \cite{K} and Hamiltonian formalism in the theory of nonlinear evolution equations (see \cite{BDK}).

A Lie conformal algebra is said to be finite if it is finitely generated  as a $\mathbb{C}[\partial]$-module. Otherwise, it is called infinite. The structure theory (see \cite{DK1}), cohomology theory (see \cite{BKV}) and representation theory (see \cite{CK1})
of finite Lie conformal algebras have been well-developed. Finite simple Lie conformal algebras are classified in \cite{CK1}, which shows that a finite simple Lie conformal algebra is either isomorphic to the Virasoro Lie conformal algebra or the current Lie conformal algebra $\text{Cur}(\mathfrak{g})$ associated to a finite-dimensional simple Lie algebra $\mathfrak{g}$. All irreducible finite conformal modules of finite simple Lie conformal algebras are classified in \cite{CK1}, the extensions of these conformal modules are studied in \cite{CK2}
and cohomology groups of finite simple Lie conformal algebras with some conformal modules
are characterized in \cite{BKV}. However, to the best of our knowledge, there is a little progression on the study of structure theory, representation theory and cohomology theory of infinite simple Lie conformal algebras.
In  previous works, the main studying object in infinite simple Lie conformal algebras is the general Lie conformal algebra $gc_N$ (see \cite{BKL1, BKL2, DK2, S, SY1}), which plays the same important role in the theory of Lie conformal algebras as the general Lie algebra $gl_N$ does in that of Lie algebras. Besides $gc_N$, these are few examples of infinite simple Lie conformal algebras.

Recently, in \cite{HW}, the authors provide a method to construct infinite simple Lie conformal algebras using the relation between quadratic Lie conformal algebras and Gel'fand-Dorfman bialgebras (see \cite{X1} or \cite{GD}). Moreover, they present three classes of infinite simple Lie conformal algebras in \cite{HW} which are obtained from the corresponding Gel'fand-Dorfman bialgebras. Therefore, for enriching the theory of infinite simple Lie conformal algebras, it is natural to investigate the structure theory and representation theory of these infinite simple Lie conformal algebras. The first class is $CL_1(c)$ which can be seen as some generalization of the graded Lie conformal algebra of $gc_1$ whose finite representation is studied in \cite{SY1}. The second class can be seen as subalgebras of the third class, and their structure theory and representation theory may be similar. In this paper, we plan to investigate some structure theory and representation theory of the second class of infinite simple Lie conformal algebras $CL(b,\varphi)$.
For the definition of $CL(b,\varphi)$, one can refer to Section 2.
Note that when $\Delta=Z$, the central extensions, conformal derivations and free conformal modules of rank one of $CL(b,0)$ are studied in \cite{FCH}. In this paper,
we plan to characterize central extensions, conformal derivations, finite conformal modules and free intermediate series modules of $CL(b,\varphi)$.
Moreover, it should be pointed out that  the coefficient algebra $\text{Coeff}(CL(b,\varphi))$ of $CL(b,\varphi)$ can be regarded as some generalization of Block type Lie algebras (see \cite{Bl, DZ}).

This paper is organized as follows. In Section 2, the definitions of Lie conformal algebra, quadratic Lie conformal algebra and $CL(b,\varphi)$ are recalled.
In Section 3, the central extensions of $CL(b,\varphi)$ by a one-dimensional center $\mathbb{C}\mathfrak{c}$ are characterized. Using this result, we obtain some central extensions of the infinite-dimensional Lie algebra $\text{Coeff}(CL(b,\varphi))$. In Section 4, we show that all conformal derivations of $CL(b,\varphi)$ are inner. In Section 5,
it is proved that $CL(b,\varphi)$ does not have any non-trivial finite conformal modules, and this provides an example of
an infinite simple Lie conformal algebra that cannot be embedded into $gc_N$ for any $N$. In Section 6, we give a classification of free intermediate series modules of $CL(b,\varphi)$.

Throughout this paper, we denote by $\mathbb{C}$ the field of complex
numbers, by $\mathbb{N}$ the set of natural numbers (i.e.
$\mathbb{N}=\{0, 1, 2,\cdots\}$) and by $\mathbb{Z}$ the set of integer
numbers. Let $\Delta$ be an infinite additive subgroup of $\mathbb{C}$ and $\mathbb{C}^+$ be the additive group of $\mathbb{C}$. All tensors over $\mathbb{C}$ are denoted by $\otimes$.
Moreover, if $A$ is a vector space, then the space of polynomials of $\lambda$ with coefficients in $A$ is denoted by $A[\lambda]$.

\section{Preliminaries}
In this section, we recall some definitions and results about Lie conformal algebras. The interested reader may consult  \cite{K1}
for related  facts.
\begin{defi}
A \emph{Lie conformal algebra} $R$ is a $\mathbb{C}[\partial]$-module with a $\lambda$-bracket $[\cdot_\lambda \cdot]$ which defines a $\mathbb{C}$-bilinear
map from $R\times R\rightarrow R[\lambda]$ satisfying
\begin{eqnarray*}
&&[\partial a_\lambda b]=-\lambda [a_\lambda b],~~~[a_\lambda \partial b]=(\lambda+\partial)[a_\lambda b], ~~\text{(conformal sesquilinearity)}\\
&&[a_\lambda b]=-[b_{-\lambda-\partial}a],~~~~\text{(skew-symmetry)}\\
&&[a_\lambda[b_\mu c]]=[[a_\lambda b]_{\lambda+\mu} c]+[b_\mu[a_\lambda c]],~~~~~~\text{(Jacobi identity)}
\end{eqnarray*}
for $a$, $b$, $c\in R$.
\end{defi}

Here, $[a_\lambda b]=\sum_{n=0}^\infty\frac{\lambda^{n}}{n!}(a_{(n)}b)$, where $a_{(n)}b$ is called the \emph{$n$-th product} of $a$ and $b$. If a Lie conformal algebra $R$ is a finitely generated
$\mathbb{C}[\partial]$-module, then it is called \emph{finite}; otherwise, it is said to be \emph{infinite}.

Moreover, there is an important Lie algebra associated with a Lie conformal algebra.
Assume that $R$ is a Lie conformal algebra. Let Coeff$(R)$ be the quotient
of the vector space with basis $a_n$ $(a\in R, n\in\mathbb{Z})$ by
the subspace spanned over $\mathbb{C}$ by
elements:
$$(\alpha a)_n-\alpha a_n,~~(a+b)_n-a_n-b_n,~~(\partial
a)_n+na_{n-1},~~~\text{where}~~a,~~b\in R,~~\alpha\in \mathbb{C},~~n\in
\mathbb{Z}.$$ The operation on Coeff$(R)$ is given as follows:
\begin{equation}\label{106}
[a_m, b_n]=\sum_{j\in \mathbb{N}}\left(\begin{array}{ccc}
m\\j\end{array}\right)(a_{(j)}b)_{m+n-j}.\end{equation} Then
Coeff$(R)$ is a Lie algebra and it is called the\emph{ coefficient algebra} of $R$ (see \cite{K1}).

\begin{defi}
A \emph{module} $M$ over a Lie conformal algebra $R$ is a $\mathbb{C}[\partial]$-module endowed with a $\mathbb{C}$-bilinear map
$R\times M\longrightarrow M[\lambda]$, $(a, v)\mapsto a_\lambda v$, satisfying the following axioms
\begin{eqnarray*}
&&(\partial a)_\lambda v  = -\lambda a_\lambda v, \\
&&  a_\lambda(\partial v) = (\partial+\lambda)a_\lambda v, \\
&&  [a_{\lambda} b]_{\lambda+\mu}v =  a_\lambda(b_{\mu} v)-b_\mu(a_{\lambda} v).
\end{eqnarray*}
for $a, b\in R, v\in M$.
If a module of $R$ is a finitely generated
$\mathbb{C}[\partial]$-module, then it is called \emph{finite}; otherwise, it is said to be \emph{infinite}.
\end{defi}

Given a module $M$ over a Lie conformal algebra $R$, we can naturally obtain a module of Coeff$(R)$. Set $a_\lambda v=\sum_{n=0}^\infty\frac{\lambda^{n}}{n!}(a_{(n)}v)$. Let $E(M)$ be the quotient
of the vector space with basis $v_n$ $(v\in M, n\in\mathbb{Z})$ by
the subspace spanned over $\mathbb{C}$ by
elements:
$$(\alpha v)_n-\alpha v_n,~~(u+v)_n-u_n-v_n,~~(\partial
v)_n+nv_{n-1},~~~\text{where}~~u,~~v\in R,~~\alpha\in \mathbb{C},~~n\in
\mathbb{Z}.$$
The operation of Coeff$(R)$ on $E(M)$ is given as follows:
\begin{equation}\label{106}
a_m v_n=\sum_{j\in \mathbb{N}}\left(\begin{array}{ccc}
m\\j\end{array}\right)(a_{(j)}v)_{m+n-j},
\end{equation}
which endows $E(M)$ as a module over Coeff$(R)$.

\begin{defi}
A Lie conformal algebra $R$ is $\Delta$-graded if $R=\oplus_{\alpha\in \Delta}R_\alpha$, where each $R_\alpha$ is a $\mathbb{C}[\partial]$-submodule and $[{R_\alpha}_\lambda R_\beta]\subset R_{\alpha+\beta}[\lambda]$ for any $\alpha$,
$\beta\in \Delta$.

Similarly, an $R$-module $V$ is called $\Delta$-graded if $V=\oplus_{\alpha\in \Delta}V_\alpha$, where each $V_\alpha$ is a $\mathbb{C}[\partial]$-submodule and ${R_\alpha}_\lambda V_\beta \subset V_{\alpha+\beta}[\lambda]$ for any $\alpha$,
$\beta\in \Delta$. In addition, if each $V_\alpha$
can be generated by one element $v_\alpha \in V_\alpha$ over $\mathbb{C}[\partial]$, we call $V$ an intermediate series
module of $R$. An intermediate series $R$-module $V$ is called free if each $V_\alpha$ is freely generated
by some $v_\alpha\in V_\alpha$ over $\mathbb{C}[\partial]$.

\end{defi}
\begin{defi}
 For a Lie conformal algebra $R$, if there exists a vector space $V$ such that $R=\mathbb{C}[\partial]V$ is a free
$\mathbb{C}[\partial]$-module and the $\lambda$-bracket is of the following form:
\begin{eqnarray*}
[a_{\lambda} b]=\partial u+\lambda v+ w,~~~~~~\text{$a$, $b\in V$,}
\end{eqnarray*}
where $u$, $v$, $w\in V$, then $R$ is called a \emph{quadratic Lie conformal algebra}.
\end{defi}

\begin{defi}(see \cite{GD} or \cite{X1})
A \emph{Gel'fand-Dorfman bialgebra} $V$ is a Lie algebra $(V,[\cdot,\cdot])$ with a binary operation $\circ$ such that $(V,\circ)$ forms a Novikov algebra where the operation $\circ$ satisfies the following conditions
\begin{eqnarray}
&&(a\circ b)\circ c-a\circ (b\circ c)=(b\circ a)\circ c-b\circ (a\circ c),\\
&&(a\circ b)\circ c=(a\circ c)\circ b,
\end{eqnarray}
and the following compatibility condition holds:
\begin{eqnarray}\label{eqq3}
[a\circ b, c]-[a\circ c, b]+[a,b]\circ c-[a,c]\circ b-a\circ [b,c]=0,
\end{eqnarray}
for $a$, $b$, and $c\in V$. We usually denote it by $(V,\circ,[\cdot,\cdot])$.
\end{defi}

An equivalent characterization of quadratic Lie conformal algebras is given as follows.
\begin{thm}(see \cite{GD} or \cite{X1})
$R=\mathbb{C}[\partial]V$ is a quadratic Lie conformal algebra if and only if the $\lambda$-bracket of $R$ is given as
follows
$$[a_{\lambda} b]=\partial(b\circ a)+[b, a]+\lambda(b\circ a+a\circ b),\qquad \text{$a$, $b\in V$},$$
and $(V, \circ, [\cdot,\cdot])$ is a Gel'fand-Dorfman bialgebra. Therefore, $R$ is called the quadratic Lie conformal algebra corresponding to
the Gel'fand-Dorfman bialgebra $(V, \circ, [\cdot,\cdot])$.
\end{thm}

Finally, we introduce the studying object in this paper.
\begin{defi}
$CL(b,\varphi)=\oplus_{\alpha\in \Delta} \mathbb{C}[\partial]x_\alpha$ is a Lie conformal algebra with the $\lambda$-bracket:
\begin{gather}
[{x_\alpha}_\lambda x_\beta]=((\alpha+b)\partial+(\alpha+\beta+2b)\lambda)x_{\alpha+\beta}\nonumber\\
+\frac{1}{b}(\varphi(\alpha)\beta-\varphi(\beta)\alpha+b(\varphi(\alpha)-\varphi(\beta)))x_{\alpha+\beta},
\end{gather}
where $\alpha, \beta\in \Delta$, $2b\notin \Delta$ and $\varphi:\Delta\rightarrow \mathbb{C}^+$ is a group homomorphism.
\end{defi}
 Note that in what follows, when we write $\varphi(3b)$, it means that $3b\in \Delta$.

By the discussion in \cite{HW}, $CL(b,\varphi)$ is a class of infinite simple Lie conformal algebras.

\begin{rmk}
Obviously, $CL(b,\varphi)$ is a quadratic Lie conformal algebra corresponding to the Gel'fand-Dorfman bialgebra $(V=\oplus_{\alpha\in \Delta}\mathbb{C}x_\alpha, \circ, [\cdot,\cdot])$ with the Novikov algebra operation and Lie bracket as follows
\begin{eqnarray}
&&{x_\alpha}\circ x_\beta=(\beta+b)x_{\alpha+\beta},\\
&&[x_\alpha, x_\beta]=\frac{1}{b}(\varphi(\beta)\alpha-\varphi(\alpha)\beta+b(\varphi(\beta)-\varphi(\alpha)))x_{\alpha+\beta}.
\end{eqnarray}

$\text{Coeff}(CL(b,\varphi))$ has a basis $\{x_{\alpha,i}|\alpha\in \Delta, i\in \mathbb{Z}\}$ and
the Lie bracket is given by
\begin{eqnarray*}
&&[x_{\alpha,i},x_{\beta,j}]=(i(\beta+b)-j(\alpha+b))x_{\alpha+\beta,i+j-1}\\
&&+\frac{1}{b}(\varphi(\alpha)\beta-\varphi(\beta)\alpha+b(\varphi(\alpha)-\varphi(\beta)))x_{\alpha+\beta,i+j}.
\end{eqnarray*}
\end{rmk}
\begin{rmk}
When $\Delta=\mathbb{Z}$, there is a Lie conformal algebra $CL=\oplus_{\alpha\in \mathbb{Z}}\mathbb{C}x_\alpha$ with the following $\lambda$-bracket
\begin{gather}
[{x_\alpha}_\lambda x_\beta]=(\alpha\partial+(\alpha+\beta)\lambda)x_{\alpha+\beta},
\end{gather}
for any $\alpha$, $\beta\in \mathbb{Z}$. Note that this Lie conformal algebra is not simple and the free intermediate series modules of $CL$ are studied in \cite{GXY}.
\end{rmk}

\section{Central extensions}
In this section, we will study central extensions of $CL(b,\varphi)$.

An extension of a Lie conformal algebra $R$ by an abelian Lie conformal algebra $C$ is a short exact sequence of Lie conformal algebras
\begin{eqnarray*}
0\rightarrow C\rightarrow \widehat{R}\rightarrow R\rightarrow 0.
\end{eqnarray*}
In this case, $\widehat{R}$ is called an \emph{extension} of $R$ by $C$. This extension is \emph{central} if $\partial C=0$ and $[C_\lambda \widehat{R}]=0$.

In the following, we focus on the central extension $\widehat{R}$ of $R$ by a one-dimensional center $\mathbb{C}\mathfrak{c}$. This implies that $\widehat{R}=R\oplus \mathbb{C}\mathfrak{c}$, and
\begin{eqnarray*}
[a_\lambda b]_{\widehat{R}}=[a_\lambda b]_R+\alpha_\lambda(a,b)\mathfrak{c}~~~\qquad\text{for~~all~~$a$, $b\in R$,}
\end{eqnarray*}
where $\alpha_\lambda(\cdot,\cdot): R\times R\rightarrow \mathbb{C}[\lambda]$ is a $\mathbb{C}$-bilinear map. By the axioms of Lie conformal algebra, $\alpha_\lambda$ should satisfy the following properties (for all $a$, $b$, $c\in R$) :
\begin{eqnarray}
\label{c1}&&\alpha_\lambda(\partial a,b)=-\lambda \alpha_\lambda(a,b)=-\alpha_\lambda(a,\partial b),\\
\label{c2}&&\alpha_\lambda(a,b)=-\alpha_{-\lambda}(b,a),\\
\label{c3}&&\alpha_\lambda(a,[b_\mu c])-\alpha_\mu(b,[a_\lambda c])=\alpha_{\lambda+\mu}([a_\lambda b],c).
\end{eqnarray}

\begin{pro}\label{t1}(see Theorem 3.1 in \cite{H})
Let $\widehat{R}=R\oplus\mathbb{C}\mathfrak{c}$ be a central extension of a quadratic Lie conformal algebra $R=\mathbb{C}[\partial]V$  corresponding to $(V,\circ, [\cdot,\cdot])$ by a one-dimensional center $\mathbb{C}\mathfrak{c}$.
Set the $\lambda$-bracket of $\widehat{R}$ by
\begin{eqnarray}
\label{e2}\widetilde{[a_\lambda b]}=\partial(b\circ a)+\lambda(a\ast b)+[b,a]+\alpha_\lambda(a,b)\mathfrak{c},
\end{eqnarray}
where $a$, $b\in V$, $a\ast b=a\circ b+b\circ a$ and $\alpha_\lambda(a,b)\in \mathbb{C}[\lambda]$. Assume that $\alpha_\lambda(a,b)=\sum_{i=0}^n\lambda^i\alpha_i(a,b)$ for any $a$, $b\in V$, in which there exist some $a$, $b \in V$ such that
$\alpha_n(a,b)\neq 0$. Then we obtain, for any $a$, $b$, $c\in V$,\\
(1) If $n>3$,  $\alpha_n(a\circ b, c)=0$ ;\\
(2) If $n\leq 3$,
\begin{eqnarray}
\label{eqqq1}&&\alpha_i(a,b)=(-1)^{i+1}\alpha_i(b,a), \text{ for any $i\in\{0,1,2,3\}$}, \\
\label{eqqe1}&&\alpha_3(a,c\circ b)=\alpha_3(a\circ b,c)=\alpha_3(b\circ a, c),\\
\label{eqq2}&&\alpha_2(a,c\circ b)+\alpha_3(a,[c,b])=\alpha_2(a\circ b,c)+\alpha_3([b,a],c),\\
\label{eqq3}&&\alpha_2(a,b\ast c)+\alpha_2(b\circ a,c)=2\alpha_2(a\circ b,c)+3\alpha_3([b,a],c),\\
\label{eqq4}&&\alpha_1(a,c\circ b)+\alpha_2(a,[c,b])=\alpha_1(a\circ b,c)+\alpha_2([b,a],c),\\
\label{eqq5}&&\alpha_1(a,b\ast c)-\alpha_1(b,a\ast c)=-\alpha_1(b\circ a,c)+\alpha_1(a\circ b,c)+2\alpha_2([b,a],c),\\
\label{eqq6}&&\alpha_0(a,c\circ b)+\alpha_1(a,[c,b])-\alpha_0(b, a\ast c)=\alpha_0(a\circ b,c)+\alpha_1([b,a],c),\\
\label{eqq7}&&\alpha_0(a,[c,b])-\alpha_0(b,[c,a])=\alpha_0([b,a],c);
\end{eqnarray}
(3) Such two cocycles $\alpha_\lambda(\cdot,\cdot)$ and $\alpha_\lambda^{'}(\cdot,\cdot)$ are equivalent if and only if there exists a linear map $\Phi:V\rightarrow \mathbb{C}$ such that
\begin{eqnarray}
\alpha_\lambda(a,b)=\alpha_\lambda^{'}(a,b)+\lambda \Phi(a\ast b)+\Phi([b,a]),~~~~\text{for~~all~~$a$, $b\in V$.}
\end{eqnarray}
\end{pro}
\begin{cor}\label{co1}
Let $V=\oplus_{\alpha\in \Delta}\mathbb{C}x_\alpha$. Consider the  central extension
$\widehat{CL(b,\varphi)}=CL(b,\varphi)\oplus\mathbb{C}\mathfrak{c}$ of $CL(b,\varphi)$
with the $\lambda$-bracket given by
$\widetilde{[{x_\alpha}_\lambda {x_\beta}]}=[{x_\alpha}_\lambda x_{\beta}]+\alpha_\lambda(x_\alpha,x_\beta)\mathfrak{c}$ for any
$\alpha$, $\beta\in\Delta$.
Then we have
\[
\alpha_\lambda(x_\alpha,x_\beta)=\sum_{i=0}^3\lambda^i\alpha_i(x_\alpha,x_\beta)
\]
for any $\alpha$, $\beta\in \Delta$, where $\alpha_i(\cdot,\cdot):V\times V\rightarrow \mathbb{C}$ are bilinear forms for any $i\in\{0,1,2,3\}$ and they satisfy (\ref{eqqq1})-(\ref{eqq7}).
\end{cor}
\begin{proof}
For any $\alpha$, $\beta\in \Delta$, set $\alpha_\lambda(x_\alpha,x_\beta)=\sum_{i=0}^{n_{\alpha,\beta}}\lambda^i\alpha_i(x_\alpha,x_\beta)$ where
$\alpha_i(\cdot,\cdot)$ are bilinear forms on $V$ and $n_{\alpha,\beta}$ is a non-negative integer depending on $\alpha$ and $\beta$. Setting $a=x_\alpha$, $b=x_\beta$ and $c=x_\gamma$ in (\ref{c3}) and by (\ref{c1}) and (\ref{c2}), we get
\begin{gather}
\label{cee}\lambda\alpha_\lambda(x_\alpha,x_\gamma\circ x_\beta)+\mu\alpha_\lambda(x_\alpha, x_\beta\ast x_\gamma)+\alpha_\lambda(x_\alpha, [x_\gamma,x_\beta])
-\mu\alpha_\mu(x_\beta,x_\gamma\circ x_\alpha)\\
-\lambda\alpha_\mu(x_\beta,x_\alpha\ast x_\gamma)-\alpha_\mu(x_\beta,[x_\gamma,x_\alpha])\nonumber\\
=(-\lambda-\mu)\alpha_{\lambda+\mu}(x_\beta\circ x_\alpha, x_\gamma)+\lambda \alpha_{\lambda+\mu}(x_\alpha\ast x_\beta,x_\gamma)+\alpha_{\lambda+\mu}([x_\beta,x_\alpha],x_\gamma).\nonumber
\end{gather}
For fixed $\alpha$, $\beta$, $\gamma$, there are only finite elements of $V$ appearing in $\alpha_\lambda(\cdot,\cdot)$ in (\ref{cee}). Therefore, we may assume the degrees of all $\alpha_\lambda(\cdot,\cdot)$ in (\ref{cee}) are smaller than some non-negative integer. So, we set $\alpha_\lambda(x_\alpha,x_\gamma\circ x_\beta)
=\sum_{i=0}^n\lambda^i\alpha_i(x_\alpha, x_\gamma\circ x_\beta)$, $\cdots$ and
$\alpha_{\lambda+\mu}(x_\alpha\ast x_\beta,x_\gamma)=\sum_{i=0}^n(\lambda+\mu)^i\alpha_i(x_\alpha\ast x_\beta,x_\gamma)$. Of course, here, $n$ depends on
$\alpha$, $\beta$ and $\gamma$.

If $n>3$, by comparing the coefficients of $\lambda^2\mu^{n-1}$ and $\lambda^{n-1}\mu^2$ in (\ref{cee}), we get
\begin{eqnarray*}
n\alpha_n(x_\alpha\circ x_\beta,x_\gamma)-C_n^2\alpha_n(x_\beta\circ x_\alpha,x_\gamma)=0,\\
C_n^2\alpha_n(x_\alpha\circ x_\beta,x_\gamma)-n\alpha_n(x_\beta\circ x_\alpha,x_\gamma)=0.
\end{eqnarray*}
Therefore, $\alpha_n(x_\alpha\circ x_\beta,x_\gamma)=\alpha_n(x_\beta\circ x_\alpha,x_\gamma)=0$. Since
 $x_\alpha\circ x_\beta=(\beta+b)x_{\alpha+\beta}$, we get $\alpha_n(x_{\alpha+\beta},x_\gamma)=0$. Thus, $\alpha_n([x_\beta,x_\alpha],x_\gamma)=0$. Repeating this process, we will have
$\alpha_m(x_\alpha\circ x_\beta,x_\gamma)=\alpha_m(x_\beta\circ x_\alpha,x_\gamma)=0$ for all $n\geq m> 3$.

According to the discussion above, for any $\alpha$, $\beta$, $\gamma\in \Delta$, we arrive at
$\alpha_m(x_\alpha\circ x_\beta,x_\gamma)=0$ for all $m>3$, implying that $\alpha_m(x_{\alpha+\beta},x_\gamma)=0$ for all $m>3$. As a result,
$\alpha_m(x,c)=0$ for all $m>3$ and any $x$ and $c\in V$. Consequently, $\alpha_\lambda(a,b)=\sum_{i=0}^3 \lambda^i\alpha_i(a,b)$, and the proof of this corollary can be completed by Proposition \ref{t1}.

\end{proof}

\begin{thm}\label{t2}
When  $3b\notin \Delta$ or $\varphi(3b)\neq 0$, there are no non-trivial  central extensions of $CL(b,\varphi)$ by a one-dimensional center $\mathbb{C}\mathfrak{c}$.

When $\varphi(3b)= 0$,
all equivalent classes of central extension of $CL(b,\varphi)$ by a one-dimensional center $\mathbb{C}\mathfrak{c}$
are $CL(b,\varphi)(g)$ with the following non-trivial $\lambda$-brackets
\begin{gather}
[{x_\alpha}_\lambda x_\beta]=((\alpha+b)\partial+(\alpha+\beta+2b)\lambda)x_{\alpha+\beta}+\frac{1}{b}(\varphi(\alpha)\beta-\varphi(\beta)\alpha+b(\varphi(\alpha)-\varphi(\beta)))x_{\alpha+\beta}\nonumber,\\
 \text{if $\alpha+\beta\neq -3b$},\nonumber\\
 [{x_\alpha}_\lambda x_{-3b-\alpha}]=((\alpha+b)\partial-b\lambda-\varphi(\alpha))x_{-3b}
+g(\alpha) \mathfrak{c},
\end{gather}
for some group homomorphism $g:\Delta \rightarrow \mathbb{C}^{+}$ satisfying $g(3b)=0$. Moreover, $CL(b,\varphi)(g)$ is equivalent to $CL(b,\varphi)(g^{'})$ if and only if there exists some $k\in \mathbb{C}$ such that $g=g^{'}+k\varphi$.
\end{thm}

\begin{proof}
By Corollary \ref{co1}, we only need to determine $\alpha_i(\cdot,\cdot):V\times V\rightarrow \mathbb{C}$ satisfying (\ref{eqqq1})-(\ref{eqq7}).

Letting $a=x_\alpha$, $c=x_\beta$ and $b=x_\gamma$ in (\ref{eqqe1}), we can directly obtain
$(\gamma-\alpha)\alpha_3(x_{\alpha+\gamma},x_\beta)=0$. According to that $\Delta$ is an infinite additive subgroup of $\mathbb{C}$,
for any $\alpha^{'}\in \Delta$, we can choose different $\alpha$, $\gamma\in \Delta$ such that $\alpha+\gamma=\alpha^{'}$. Therefore, we have
$\alpha_3(x_{\alpha^{'}},x_\beta)=0$ for any $\alpha^{'}$ and $\beta\in \Delta$.

Similarly, by (\ref{eqq2}) and (\ref{eqq3}), one can get
\begin{eqnarray}
\label{f1}(\gamma+b)(\alpha_2(x_\alpha,x_{\beta+\gamma})-\alpha_2(x_{\alpha+\gamma},x_\beta))=0,\\
\label{f2}(\beta+\gamma+2b)\alpha_2(x_\alpha,x_{\beta+\gamma})=(2\gamma-\alpha+b)\alpha_2(x_{\alpha+\gamma},x_\beta).
\end{eqnarray}
Since $b\notin \Delta$, $\gamma+b\notin \Delta$. Therefore, $\gamma+b\neq 0$. By (\ref{f1}), we obtain
$\alpha_2(x_\alpha,x_{\beta+\gamma})=\alpha_2(x_{\alpha+\gamma},x_\beta)$. Thus, applying it to (\ref{f2}),
we get $(\alpha+\beta-\gamma+b)\alpha_2(x_{\alpha+\gamma},x_\beta)=0$. Since $\alpha+\beta-\gamma+b\neq 0$,
we can immediately have $\alpha_2(x_\alpha,x_\beta)=0$ for any $\alpha$, $\beta\in \Delta$.

By (\ref{eqq4}), we have $(\gamma+b)(\alpha_1(x_\alpha,x_{\beta+\gamma})-\alpha_1(x_{\alpha+\gamma},x_\beta))=0$.
Therefore,
\begin{eqnarray}
\label{f3}\alpha_1(x_\alpha,x_{\beta+\gamma})=\alpha_1(x_{\alpha+\gamma},x_\beta).
\end{eqnarray}
Letting $\alpha=0$ in (\ref{f3}), one can obtain $\alpha_1(x_{\gamma},x_\beta)=\alpha_1(x_0,x_{\beta+\gamma})$. Therefore, we can assume that $\alpha_1(x_{\alpha},x_\beta)=f(\alpha+\beta)$ where $f$ is a map from $\Delta$ to $\mathbb{C}$. Taking it into (\ref{eqq5}), (\ref{eqq5}) naturally holds. Therefore, $\alpha_1(x_{\alpha},x_\beta)=f(\alpha+\beta)$ for any
$\alpha$, $\beta\in \Delta$. By Proposition \ref{t1}, define $\Phi: V\rightarrow \mathbb{C}$ by
$\Phi(x_\alpha)=\frac{f(\alpha)}{\alpha+2b}$, which allows us to assume that $\alpha_1(\cdot,\cdot)$ be zero.

According to (\ref{eqq6}) and (\ref{eqq7}), we get
\begin{gather}
\label{f4}(\gamma+b)(\alpha_0(x_\alpha,x_{\beta+\gamma})-\alpha_0(x_{\alpha+\gamma},x_\beta))-(\alpha+\beta+2b)\alpha_0(x_\gamma,x_{\alpha+\beta})
=0,\\
(\varphi(\gamma)\beta-\varphi(\beta)\gamma+b(\varphi(\gamma)-\varphi(\beta)))\alpha_0(x_\alpha,x_{\beta+\gamma})\nonumber\\
-(\varphi(\alpha)\beta-\varphi(\beta)\alpha+b(\varphi(\alpha)-\varphi(\beta)))\alpha_0(x_\gamma,x_{\beta+\alpha})\nonumber\\
\label{f5}=(\varphi(\alpha)\gamma-\varphi(\gamma)\alpha+b(\varphi(\alpha)-\varphi(\gamma)))\alpha_0(x_{\alpha+\gamma},x_{\beta}).
\end{gather}

Letting $\gamma=0$ in (\ref{f4}),we have $(\alpha+\beta+2b)\alpha_0(x_0,x_{\alpha+\beta})=0$. Since $2b\notin \Delta$, $\alpha_0(x_0,x_{\alpha})=0$ for any $\alpha\in\Delta$. Setting $\alpha=0$ in (\ref{f4}), we obtain
$(\beta+\gamma+3b)\alpha_0(x_\gamma,x_\beta)=(\gamma+b)\alpha_0(x_0,x_{\beta+\gamma})=0$. Therefore, if
$3b\in \Delta$, due to $\alpha_0(x_\alpha,x_\beta)=-\alpha_0(x_\beta,x_\alpha)$, $\alpha_0(x_\gamma,x_\beta)=\delta_{\beta+\gamma,-3b}g(\beta)$ where $g(x)$ is a complex function on $\Delta$ , $g(\beta)=-g(-3b-\beta)$ and $g(0)=g(-3b)=0$; if $3b\notin \Delta$, $\alpha_0(x_\gamma,x_\beta)=0$. Then if $\alpha+\beta+\gamma=-3b$, it can be directly obtained from (\ref{f4}) that $g(\alpha)+g(\beta)+g(\gamma)=0$. Therefore, $g(\alpha)+g(\beta)+g(\gamma)=g(-3b-\beta-\gamma)+g(\beta)+g(\gamma)=-g(\beta+\gamma)+g(\beta)+g(\gamma)=0$ for any $\beta$, $\gamma\in \Delta$.
Therefore, in this case, if $3b\in \Delta$, then $\alpha_0(x_\gamma,x_\beta)=\delta_{\beta+\gamma,-3b}g(\beta)$, where $g:\Delta \rightarrow \mathbb{C}^{+}$ is a  group homomorphism  and $g(3b)=0$; if $3b\notin \Delta$, we get $\alpha_0(x_\alpha,x_\beta)=0$. Then (\ref{f5}) holds. By (3) in Proposition \ref{t1}, $g(\alpha)$ is equivalent to $g^{'}(\alpha)$ if and only if there exists some $k\in \mathbb{C}$ such that $g=g^{'}+k\varphi$.

Therefore, this theorem follows by the above discussion.
\end{proof}
\begin{cor}
When $\varphi(3b)= 0$, for some given group homomorphism $g:\Delta \rightarrow \mathbb{C}^{+}$ satisfying $g(3b)=0$,
there are the following  central extension of $\text{Coeff}(CL(b,\varphi))$ by a one-dimensional center $\mathbb{C}\mathfrak{c}_{-1}$ with the following non-trivial Lie brackets
\begin{gather}
[x_{\alpha,i},x_{\beta,j}]=(i(\beta+b)-j(\alpha+b))x_{\alpha+\beta,i+j-1}\nonumber\\
+\frac{1}{b}(\varphi(\alpha)\beta-\varphi(\beta)\alpha+b(\varphi(\alpha)-\varphi(\beta)))x_{\alpha+\beta,i+j},\\
 \text{if $\alpha+\beta\neq -3b$ and }\nonumber\\
 [x_{\alpha,i},x_{-3b-\alpha,j}]=(i(-2b-\alpha)-j(\alpha+b))x_{-3b,i+j-1}-\varphi(\alpha))x_{-3b,i+j}
+g(\alpha)\delta_{i+j,-1}\mathfrak{c}_{-1}.
\end{gather}
\end{cor}
\begin{proof}
This corollary can be directly obtained by Theorem \ref{t2} and the relation between Lie conformal algebras and their coefficient algebras.
\end{proof}

\section{Conformal derivations}
In this section, we study the conformal derivations of $CL(b,\varphi)$.

\begin{defi}
A \emph{conformal linear map} between $\mathbb{C}[\partial]$-modules $U$ and $V$ is a linear map
$\phi_{\lambda}: U\rightarrow V[\lambda]$ such that
\begin{eqnarray}
\phi_\lambda(\partial u)=(\partial+\lambda)\phi_\lambda u, ~~\text{for~~all~~$u\in U$.}
\end{eqnarray}
\end{defi}

\begin{defi}
Let $R$ be a Lie conformal algebra. A conformal linear map $d_\lambda: R\rightarrow R[\lambda]$ is
called a \emph{conformal derivation} of $R$ if
\begin{eqnarray}\label{d1}
d_\lambda[a_\mu b]=[(d_\lambda a)_{\lambda+\mu} b]+[a_\mu(d_\lambda b)], ~~\text{ $a$, $b\in R$.}
\end{eqnarray}
\end{defi}

The space of all conformal derivations of $R$ is denoted by $\text{CDer}(R)$. For any $a\in R$, there is a natural conformal linear map $\text{ad}(a)_\lambda: R\rightarrow R[\lambda]$ such that $$(\text{ad} ~(a))_\lambda b=[a_\lambda b],~~~b\in R.$$
All conformal derivations of this kind are called \emph{inner}. The space of all inner conformal derivations is denoted by
$\text{CInn}(R)$.

Set $D_\lambda \in \text{CDer}(CL(b,\varphi))$. Define $D^\alpha_\lambda(x_\beta)=\pi_{\alpha+\beta}(D_\lambda(x_\beta))$, where
$\pi_\alpha$ is the natural projection from $\mathbb{C}[\lambda]\otimes CL(b,\varphi)\cong \oplus_{\beta\in \Delta}\mathbb{C}[\lambda,\partial]x_\beta$ onto $\mathbb{C}[\lambda,\partial]x_\alpha$.
Then $D^\alpha_\lambda$ is a conformal derivation and $D_\lambda=\sum_{\alpha\in \Delta}D^\alpha_\lambda$ in the sense that for any $y\in CL(b,\varphi)$, there  are only finitely many $D_\lambda^\alpha(y)\neq 0$.

\begin{lem}\label{lt1}
For any $\alpha\in\Delta$, $D^\alpha_\lambda$ is an inner conformal derivation of the form $D=ad(g(\partial)x_\alpha)_\lambda$ for some $g(\partial)\in \mathbb{C}[\partial]$.
\end{lem}
\begin{proof}
Set $D_\lambda^\alpha(x_\beta)=f_\beta(\lambda,\partial)x_{\alpha+\beta}$ where $f_\beta(\lambda,\partial)\in \mathbb{C}[\lambda,\partial]$. Applying $D_\lambda^\alpha$ to
$[{x_0}_\mu x_\beta]=(b\partial+(\beta+2b)\mu-\varphi(\beta))x_\beta$,
we can get
\begin{gather}
(b(\lambda+\partial)+(\beta+2b)\mu-\varphi(\beta))f_\beta(\lambda,\partial)
=f_0(\lambda,-\lambda-\mu)((\alpha+b)\partial+(\alpha+\beta+2b)(\lambda+\mu)\nonumber\\
+\frac{1}{b}(\varphi(\alpha)\beta-\varphi(\beta)\alpha
\label{o1}+b(\varphi(\alpha)-\varphi(\beta)))+
f_\beta(\lambda,\mu+\partial)(b\partial+(\alpha+\beta+2b)\mu-\varphi(\alpha+\beta)).
\end{gather}
Setting $\mu=0$ in (\ref{o1}), we can obtain
\begin{gather}
(b\lambda+\varphi(\alpha))f_\beta(\lambda,\partial)=f_0(\lambda,-\lambda)
((\alpha+b)\partial+(\alpha+\beta+2b)\lambda\nonumber\\
+\frac{1}{b}(\varphi(\alpha)\beta-\varphi(\beta)\alpha+b(\varphi(\alpha)-\varphi(\beta)))).
\end{gather}
Note that $b\lambda+\varphi(\alpha)$ cannot divide $(\alpha+b)\partial+(\alpha+\beta+2b)\lambda
+\frac{1}{b}(\varphi(\alpha)\beta-\varphi(\beta)\alpha+b(\varphi(\alpha)-\varphi(\beta)))$.
Therefore, $b\lambda+\varphi(\alpha)$ can divide $f_0(\lambda,-\lambda)$.
Set $g(\lambda)=\frac{f_0(\lambda,-\lambda)}{b\lambda+\varphi(\alpha)}$. Then
$f_\beta(\lambda,\partial)=g(\lambda)((\alpha+b)\partial+(\alpha+\beta+2b)\lambda
+\frac{1}{b}(\varphi(\alpha)\beta-\varphi(\beta)\alpha+b(\varphi(\alpha)-\varphi(\beta))))$.
Therefore, $D_\lambda^\alpha=ad(g(-\partial)x_\alpha)_\lambda$.
\end{proof}
\begin{thm}
$\text{CDer}(CL(b,\varphi))=\text{CInn}(CL(b,\varphi))$, i.e. any conformal derivation of $CL(b,\varphi)$ is inner.
\end{thm}
\begin{proof}
By Lemma \ref{lt1}, we can get $D_\lambda=\sum_{\alpha\in \Delta}D^{\alpha}_\lambda
=\sum_{\alpha\in \Delta}ad(g_\alpha(\partial)x_\alpha)_\lambda$ for some $g_\alpha(\partial)\in\mathbb{C}[\partial]$.
If there are infinite many $\alpha$ such that $g_\alpha(\partial)\neq 0$, then
$D_\lambda(x_0)=\sum_{\alpha\in \Delta}g_\alpha(-\lambda)((\alpha+b)\partial+(\alpha+2b)\lambda+\varphi(\alpha))x_\alpha$
is an infinite sum. It contradicts with the definition of conformal derivation. Therefore, $D_\lambda=\sum_{\alpha\in \Delta}D^{\alpha}_\lambda
=\sum_{\alpha\in \Delta}ad(g_\alpha(\partial)x_\alpha)_\lambda$ is a finite sum. Set $g=\sum_{\alpha\in \Delta}g_\alpha(\partial)x_\alpha$. Then, $D_\lambda=ad(g)_\lambda$. Therefore, this theorem holds.
\end{proof}
\section{Finite conformal modules}
Suppose that $V$ is a finitely $\mathbb{C}[\partial]$-generated nontrivial $CL(b,\varphi)$-module.
Since $\mathbb{C}[\partial]$ is a principle ideal domain, $V$ can be decomposed as a sum of
a free $\mathbb{C}[\partial]$-module and a torsion $\mathbb{C}[\partial]$-module. According to the fact that
a torsion $\mathbb{C}[\partial]$-module must be a trivial  $CL(b,\varphi)$-module,
we can assume that $V$ is a free $\mathbb{C}[\partial]$-module.

According to that $\mathbb{C}[\partial]\frac{x_0}{b}$ is the Virasoro Lie conformal algebra, $V$ can be seen as a module over
$Vir$. By Theorem 3.2(1) in \cite{CK1}, we can give a composition series as follows
\begin{eqnarray}
V=V_m\supset V_{m-1}\supset\cdots\supset V_1\supset V_0,
\end{eqnarray}
where the composition factor $\overline{V_i}=V_i/V_{i-1}$, $i\geq 1$ is either a rank one free module $M_{\gamma_i,\alpha_i}$ with $\gamma_i\neq 0$ or a one-dimensional trivial module $\mathbb{C}_{\alpha_i}$ with trivial $\lambda$-action and scalar $\partial$-action by $\alpha_i$, and $V_0$ is a trivial $Vir$-module. We denote a $\mathbb{C}[\partial]$-generator of $\overline{V_i}$ by $\overline{v_i}$, and
the preimage of $\overline{v_i}$ by $v_i\in V_i$ and the $\mathbb{C}[\partial]$-generators of $V_0$ are $w_1$, $\cdots$, $w_r$. Therefore, $\{w_i| 1\leq i\leq r\} \cup \{v_i| 1\leq i\leq m\}$ forms a $\mathbb{C}[\partial]$-generating set of $V$ and the $\lambda$-action of $x_0$ on $w_i$ is trivial and on $v_i$ is a
$\mathbb{C}[\lambda,\partial]$-combination of $w_1$, $\cdots$, $w_r$, $v_1$, $\cdots$, $v_i$.
\begin{lem}\label{llem1}
For any nonzero number $\alpha\in \Delta$ and all $i\gg 0$, the $\lambda$-actions of $x_{i\alpha}$ on
$w_j$ are trivial for any $j\in \{1,\cdots,r\}$.
\end{lem}
\begin{proof}
First, we discuss it when ${x_{i\alpha}}_\lambda w_j\in V_0[\lambda]$.  Let ${x_0}_\mu$ act on it. We can get
${x_0}_\mu({x_{i\alpha}}_\lambda w_j)=0.$ Therefore, $[{x_0}_\mu x_{i\alpha}]_{\lambda+\mu}w_j=0$, i.e.
$(b(\mu-\lambda)+i\alpha\mu-i\varphi(\alpha)){x_{i\alpha}}_\lambda w_j=0$. It follows that ${x_{i\alpha}}_\lambda w_j=0$ for any $j$.

Next, by the above discussion, we can suppose that ${x_{i\alpha}}_\lambda w_j\neq 0$ for some fixed $i\gg 0$ and
$m_{ij}\geq 1$ be the largest integer such that 
${x_{i\alpha}}_\lambda w_j\notin V_{m_{ij}-1}[\lambda]$. We discuss it in the following two cases.

{\bf Case 1}: $\overline{V_{m_{ij}}}=M_{\gamma_{m_{ij}},\alpha_{m_{ij}}}$.
Then we have
\begin{eqnarray}
{x_{i\alpha}}_\lambda w_j=p_{ij}(\lambda,\partial)v_{m_{ij}}~~~(\text{mod}~~V_{m_{ij}-1})~~~~~\text{for some $p_{ij}(\lambda,\partial)\in
\mathbb{C}[\lambda,\partial].$}
\end{eqnarray}
Let ${x_0}_\mu$ act on it. We obtain
\begin{eqnarray}\label{wq2}
(b(\mu-\lambda)+i\alpha\mu-i\varphi(\alpha))p_{ij}(\lambda+\mu,\partial)=bp_{ij}(\lambda,\mu+\partial)(\partial+\gamma_{m_{ij}}\mu+\alpha_{m_{ij}}).
\end{eqnarray}
Setting $\mu=0$ in (\ref{wq2}), we can get $p_{ij}(\lambda,\partial)=0$.

{\bf Case 2}: $\overline{V_{m_i}}=\mathbb{C}_{\alpha_{m_i}}$. We can assume that
\begin{eqnarray}
{x_{i\alpha}}_\lambda w_j=p_{ij}(\lambda)v_{m_{ij}}~~~(\text{mod}~~V_{m_{ij}-1})~~~~~\text{for some $p_{ij}(\lambda,\partial)\in
\mathbb{C}[\lambda,\partial].$}
\end{eqnarray}
Let ${x_0}_\mu$ act on it. We get
\begin{eqnarray}
(b(\mu-\lambda)+i\alpha\mu-i\varphi(\alpha))p_{ij}(\lambda+\mu)=0.
\end{eqnarray}
Therefore, $p_{ij}(\lambda)=0$.

This lemma follows by the above discussion.
\end{proof}
\begin{lem}\label{llem2}
For any nonzero number $\alpha\in \Delta$ and all $i\gg 0$, the $\lambda$-actions of $x_{i\alpha}$ on
$v_1$ are trivial.
\end{lem}
\begin{proof}
Suppose that ${x_{i\alpha}}_\lambda v_1\neq 0$ for some fixed $i\gg 0$ and
$m_i\geq 1$ be the largest integer such that ${x_{i\alpha}}_\lambda v_1\notin V_{m_i-1}[\lambda]$.
Then we discuss it in the following cases.

{\bf Case 1}: $V_1=M_{\gamma_1,\alpha_1}$ and $\overline{V_{m_i}}=M_{\gamma_{m_i},\alpha_{m_i}}$.
Then we  get
\begin{eqnarray}\label{er1}
{x_{i\alpha}}_\lambda v_1=p_i(\lambda,\partial)v_{m_i} (\text{mod}~~V_{m_i-1})~~~~~\text{for some $p_i(\lambda,\partial)\in
\mathbb{C}[\lambda,\partial].$}
\end{eqnarray}
Let ${x_0}_\mu$ act on (\ref{er1}). We directly obtain
\begin{eqnarray}
bp_i(\lambda,\mu+\partial)(\partial+\gamma_{m_i}\mu+\alpha_{m_i})
=((i\alpha+b)\mu-b\lambda-i\varphi(\alpha))p_i(\lambda+\mu,\partial)\nonumber\\
\label{er2}+b(\partial+\lambda+\gamma_1\mu+\alpha_1)p_i(\lambda,\partial).
\end{eqnarray}
Setting $\partial=0$ in the above equality and letting $p_i(\lambda)=p_i(\lambda,0)$, one can get
\begin{eqnarray}
\label{er5}p_i(\lambda,\mu)=\frac{((i\alpha+b)\mu-b\lambda-i\varphi(\alpha))p_i(\lambda+\mu)+b(\lambda+\gamma_1\mu+\alpha_1)p_i(\lambda)}{b(\gamma_{m_i}\mu+\alpha_{m_i})}.
\end{eqnarray}
Letting $\partial=-\gamma_{m_i}\mu-\alpha_{m_i}$ and $\lambda=\frac{(i\alpha+b)\mu-i\varphi(\alpha)}{b}$ in (\ref{er2}),
we have
\begin{gather}
((\frac{i\alpha+b}{b}+\gamma_1-\gamma_{m_i})\mu-\alpha_{m_i}-\frac{i\varphi(\alpha)}{b}+\alpha_1)\nonumber\\
((i\alpha+b)((\gamma_{m_i}+1)\mu+\alpha_{m_i})p_i((\frac{i\alpha+b}{b}-\gamma_{m_i})\mu-\frac{i}{b}\varphi(\alpha)-\alpha_{m_i})\nonumber\\
-b((\frac{i\alpha+b}{b}-\gamma_1\gamma_{m_i})\mu-\frac{i}{b}\varphi(\alpha)-\gamma_1\alpha_{m_i}+\alpha_1))
p_i(\frac{i\alpha+b}{b}\mu-\frac{i}{b}\varphi(\alpha)))=0
\end{gather}
Note that when $i\gg 0$, we have $(\frac{i\alpha+b}{b}+\gamma_1-\gamma_{m_i})\neq 0$. Therefore,
\begin{eqnarray}
((i\alpha+b)((\gamma_{m_i}+1)\mu+\alpha_{m_i})p_i((\frac{i\alpha+b}{b}-\gamma_{m_i})\mu-\frac{i}{b}\varphi(\alpha)-\alpha_{m_i})\nonumber\\
\label{er4}=b((\frac{i\alpha+b}{b}-\gamma_1\gamma_{m_i})\mu-\frac{i}{b}\varphi(\alpha)-\gamma_1\alpha_{m_i}+\alpha_1)
p_i(\frac{i\alpha+b}{b}\mu-\frac{i}{b}\varphi(\alpha)).
\end{eqnarray}
Assume that the degree of $p_i(\lambda)$ is $n_i$. By comparing the coefficients of $\mu^{n_i+1}$ in (\ref{er4}), it follows that
\begin{eqnarray}
\label{er5}(i\alpha+b)(\gamma_{m_i}+1)(\frac{i\alpha+b}{b}-\gamma_{m_i})^{n_i}
=(i\alpha+b-b\gamma_1\gamma_{m_i})(\frac{i\alpha+b}{b})^{n_i}.
\end{eqnarray}
Since $i$ is sufficiently large and nonzero numbers $\gamma_1$ and $\gamma_{m_i}$ have only finitely many choices,(\ref{er5}) cannot hold if $n_i>1$. Therefore, $n_i\leq 1$. By (\ref{er5}),
we can set $p_i(\lambda,\partial)=a_{i,0}+a_{i,1}\lambda+a_{i,2}\partial$. Taking it into (\ref{er2}) and by comparing
the coefficients of $\mu\partial$, we can get
$ba_{i,2}(\gamma_{m_i}-\gamma_1-i\alpha)=0$. Since $i$ is sufficiently large and $\gamma_{m_i}$ has only finitely many choices,
$a_{i,2}=0$. Similarly, we can get $a_{i,0}=a_{i,1}=0$. Therefore, $p_i(\lambda,\partial)=0$.

{\bf Case 2}: $V_1=\mathbb{C}_{\alpha_1}$ and $\overline{V_{m_i}}=M_{\gamma_{m_i},\alpha_{m_i}}$. With the same assumption we did in (\ref{er1}), and applying ${x_0}_\mu$ into (\ref{er1}), we immediately get
\begin{eqnarray}
\label{er6}bp_i(\lambda,\mu+\partial)(\partial+\gamma_{m_i}\mu+\alpha_{m_i})
=((i\alpha+b)\mu-b\lambda-i\varphi(\alpha))p_i(\lambda+\mu,\partial).
\end{eqnarray}
Setting $\mu=\partial=0$ in (\ref{er6}), we have $p_i(\lambda,0)=0$. Therefore,
letting $\partial=0$ in (\ref{er6}), we get $p_i(\lambda,\partial)=0$.

{\bf Case 3}: $V_1=M_{\gamma_{1},\alpha_{1}}$ and $\overline{V_{m_i}}=\mathbb{C}_{\alpha_{m_i}}$.
Then we may assume that
\begin{eqnarray}\label{er7}
{x_{i\alpha}}_\lambda v_1=p_i(\lambda)v_{m_i} (\text{mod}~~V_{m_i-1})~~~~~\text{for some $p_i(\lambda)\in
\mathbb{C}[\lambda].$}
\end{eqnarray}
Letting ${x_0}_\mu$ act on (\ref{er7}), we  obtain
\begin{eqnarray}
((i\alpha+b)\mu-b\lambda-i\varphi(\alpha))p_i(\lambda+\mu)
+b(\partial+\lambda+\gamma_1\mu+\alpha_1)p_i(\lambda)=0.
\end{eqnarray}
One can directly get $p_i(\lambda)=0$ by comparing the coefficients of $\partial$.

{\bf Case 4}: $V_1=\mathbb{C}_{\alpha_{1}}$ and $\overline{V_{m_i}}=\mathbb{C}_{\alpha_{m_i}}$.
With a similar discussion as above, we can get $p_i(\lambda)=0$.

Therefore, by the discussion above, we finish the proof of this lemma.
\end{proof}

\begin{thm}
$CL(b,\varphi)$ does not have a non-trivial representation on any finite $\mathbb{C}[\partial]$-modules.
\end{thm}
\begin{proof}
By Lemmas \ref{llem1} and \ref{llem2}, we can also get ${x_{i\alpha}}_\lambda v_j=0$ by induction on $j\leq m$ for $i\gg0$ and any nonzero $\alpha\in \Delta$.
Therefore, the $\lambda$-action of $x_{i\alpha}$ is trivial. Since $CL(b,\varphi)$ is a simple Lie conformal algebra,
we can directly obtain that the $\lambda$-action of $CL(b,\varphi)$ is trivial. Then this theorem can be obtained.
\end{proof}

\section{Free intermediate series modules}
In this section, we study the free intermediate series modules of $CL(b,\varphi)$.

Set $V=\oplus_{\alpha\in \Delta}\mathbb{C}[\partial]M_\alpha$. Assume that
${x_\alpha}_\lambda M_\beta=f_{\alpha,\beta}(\lambda,\partial)M_{\alpha+\beta}$, where
$f_{\alpha,\beta}(\lambda,\partial)\in \mathbb{C}[\lambda,\partial]$.
Now we shall determine the structure constant $f_{\alpha,\beta}(\lambda,\partial)$ such that
$V$ is a free intermediate series module of $CL(b,\varphi)$.

According to
$[{x_\alpha}_\lambda x_\beta]_{\lambda+\mu}M_\gamma={x_\alpha}_\lambda({x_\beta}_\mu M_\gamma)
-{x_\beta}_\mu({x_\alpha}_\lambda M_\gamma)$, all structure constants should satisfy
\begin{gather}
((\beta+b)\lambda-(\alpha+b)\mu+\frac{\varphi(\alpha)\beta-\varphi(\beta)\alpha+b(\varphi(\alpha)-\varphi(\beta))}{b})f_{\alpha+\beta,\gamma}(\lambda+\mu,\partial)\nonumber\\
\label{i1}=f_{\beta,\gamma}(\mu,\lambda+\partial)f_{\alpha,\beta+\gamma}(\lambda,\partial)-f_{\alpha,\gamma}(\lambda,\mu+\partial)f_{\beta,\alpha+\gamma}(\mu,\partial).
\end{gather}

\begin{pro}\label{pp1}
If there exist some $\alpha_0$, $\beta_0\in \Delta$ such that $f_{\alpha_0,\beta_0}(\lambda,\partial)=0$, then
$f_{\alpha,\beta}(\lambda,\partial)=0$ for all $\alpha$, $\beta\in \Delta$.
\end{pro}
\begin{proof}
{\bf Claim 1}: If $f_{0,\beta+\gamma}(\lambda,\partial)=0$, then $f_{\beta,\gamma}(\lambda,\partial)=0$. \\
Letting $\alpha=0$ in (\ref{i1}), we can get
\begin{gather}
((\beta+b)\lambda-b\mu-\varphi(\beta))f_{\beta,\gamma}(\lambda+\mu,\partial)
=f_{\beta,\gamma}(\mu,\lambda+\partial)f_{0,\beta+\gamma}(\lambda,\partial)\nonumber\\
\label{i2}-f_{0,\gamma}(\lambda,\mu+\partial)f_{\beta,\gamma}(\mu,\partial).
\end{gather}
If $f_{0,\beta+\gamma}(\lambda,\partial)=0$, by (\ref{i2}), we get
\begin{eqnarray}
\label{i3}((\beta+b)\lambda-b\mu-\varphi(\beta))f_{\beta,\gamma}(\lambda+\mu,\partial)=-f_{0,\gamma}(\lambda,\mu+\partial)f_{\beta,\gamma}(\mu,\partial).
\end{eqnarray}
Letting $\lambda=0$ in (\ref{i3}), we can easily get $f_{\beta,\gamma}(\lambda,\partial)=0$.

{\bf Claim 2}: If $f_{\beta,\gamma}(\lambda,\partial)=0$, then $f_{\alpha,\gamma}(\lambda,\partial)=0$ for any $\alpha \in \Delta$. \\
If $f_{\beta,\gamma}(\lambda,\partial)=0$, then by (\ref{i1}), one can obtain
\begin{gather}
((\beta+b)\lambda-(\alpha+b)\mu+\frac{\varphi(\alpha)\beta-\varphi(\beta)\alpha+b(\varphi(\alpha)-\varphi(\beta))}{b})f_{\alpha+\beta,\gamma}(\lambda+\mu,\partial)\nonumber\\
\label{i4}=-f_{\alpha,\gamma}(\lambda,\mu+\partial)f_{\beta,\alpha+\gamma}(\mu,\partial).
\end{gather}
Since $b\notin \Delta$, $\beta+b\neq 0$ and $\alpha+b\neq 0$.
Therefore,
\[
(\beta+b)\lambda-(\alpha+b)\mu+\frac{\varphi(\alpha)\beta-\varphi(\beta)\alpha+b(\varphi(\alpha)-\varphi(\beta))}{b}
\]
is an irreducible polynomial. As a result, $(\beta+b)\lambda-(\alpha+b)\mu+\frac{\varphi(\alpha)\beta-\varphi(\beta)\alpha+b(\varphi(\alpha)-\varphi(\beta))}{b}$ can divide $f_{\alpha,\gamma}(\lambda,\mu+\partial)$ or $f_{\beta,\alpha+\gamma}(\mu,\partial)$.  If
\[
(\beta+b)\lambda-(\alpha+b)\mu+\frac{\varphi(\alpha)\beta-\varphi(\beta)\alpha+b(\varphi(\alpha)-\varphi(\beta))}{b} \mid f_{\beta,\alpha+\gamma}(\mu,\partial),
\]
then $f_{\beta,\alpha+\gamma}(\mu,\partial)=0$. So, by (\ref{i4}), $f_{\alpha+\beta,\gamma}(\lambda,\partial)=0$ for any $\alpha\in \Delta$. Therefore, $f_{\alpha,\gamma}(\lambda,\partial)=0$ for any $\alpha\in \Delta$. If
\[
(\beta+b)\lambda-(\alpha+b)\mu+\frac{\varphi(\alpha)\beta-\varphi(\beta)\alpha+b(\varphi(\alpha)-\varphi(\beta))}{b}|f_{\alpha,\gamma}(\lambda,\mu+\partial),
\]
then it is easy to get $f_{\alpha,\gamma}(\lambda,\partial)=0$ for any $\alpha \in \Delta$.

This proposition can be directly obtained from the above claims.
\end{proof}

In what follows, we will assume that $V$ is a non-trivial module of $CL(b,\varphi)$.
According to Proposition \ref{pp1}, all $f_{\alpha,\gamma}(\lambda,\partial)$ are non-zero.
\begin{lem}\label{le1}
For any $\gamma\in \Delta$, $f_{0,\gamma}(\lambda,\partial)=b(\partial+q_\gamma\lambda+m_\gamma)$ for some
$q_\gamma$, $m_\gamma \in\mathbb{C}$.
\end{lem}
\begin{proof}
Since $[{x_0}_\lambda  x_0]=b(\partial+2\lambda)x_0$, then
$\mathbb{C}[\partial]x_0^{'}$ is the Virasoro Lie conformal algebra by letting $x_0^{'}=\frac{x_0}{b}$.
Since $\mathbb{C}[\partial]M_\gamma$ is a non-trivial module of rank one of
$\mathbb{C}[\partial]x_0^{'}$, it follows that
\[
f_{0,\gamma}(\lambda,\partial)=b(\partial+q_\gamma\lambda+m_\gamma)
\]
for some $q_\gamma$, $m_\gamma \in\mathbb{C}$ according to the representation theory of Virasoro Lie conformal algebra, as desired.
\end{proof}

\begin{lem}\label{le2}
For any $\beta\in \Delta$, we have $m_\beta=c-\frac{\varphi(\beta)}{b}$ for some $c\in \mathbb{C}$.
\end{lem}
\begin{proof}
By Lemma \ref{le1} and (\ref{i2}), one gets
\begin{gather}
((\beta+b)\lambda-b\mu-\varphi(\beta))f_{\beta,\gamma}(\lambda+\mu,\partial)
=bf_{\beta,\gamma}(\mu,\lambda+\partial)(\partial+q_{\beta+\gamma}\lambda+m_{\beta+\gamma})\nonumber\\
\label{i5}-bf_{\beta,\gamma}(\mu,\partial)(\partial+\mu+q_\gamma\lambda+m_\gamma).
\end{gather}

Setting $\lambda=0$ in (\ref{i5}), we obtain $bf_{\beta,\gamma}(\mu,\partial)(m_{\beta+\gamma}-m_\gamma+\frac{\varphi(\beta)}{b})=0$. Since
$f_{\beta,\gamma}(\mu,\partial)\neq 0$, it follows that
\begin{eqnarray}
\label{i6}m_{\beta+\gamma}=m_\gamma-\frac{\varphi(\beta)}{b}.
\end{eqnarray}
Setting $\gamma=0$ in (\ref{i6}) and $m_0=c$, we get $m_\beta=c-\frac{\varphi(\beta)}{b}$. Then (\ref{i6}) naturally holds due to the fact that $\varphi$ is a group homomorphism.
\end{proof}

\begin{lem}\label{le3}
Set $f_{\beta,\gamma}(\lambda,\partial)=\sum_{i=0}^ma_i(\lambda)\partial^i$ for some $\beta$, $\gamma\in \Delta$
with $a_m(\lambda)\neq 0$. Then $q_{\beta+\gamma}-q_\gamma=-m+\frac{\beta+b}{b}$.
\end{lem}
\begin{proof}
Set $\mu=\frac{\beta+b}{b}\lambda-\frac{\varphi(\beta)}{b}$ in (\ref{i5}).
By Lemma \ref{le2}, one can get
\begin{eqnarray}
\label{k1}f_{\beta,\gamma}(\frac{\beta+b}{b}\lambda-\frac{\varphi(\beta)}{b},\lambda+\partial)(\partial+q_{\beta+\gamma}\lambda+m_{\beta+\gamma})
\\
=f_{\beta,\gamma}(\frac{\beta+b}{b}\lambda-\frac{\varphi(\beta)}{b},\partial)(\partial+(\frac{\beta+b}{b}+q_\gamma)\lambda+m_{\beta+\gamma}).\nonumber
\end{eqnarray}
Replacing $\frac{\beta+b}{b}\lambda-\frac{\varphi(\beta)}{b}$ by $\mu$ in (\ref{k1}), we can obtain
\begin{gather}
\label{k2}f_{\beta,\gamma}(\mu,\frac{b}{\beta+b}\mu+\frac{\varphi(\beta)}{\beta+b}+\partial)(\partial+\frac{bq_{\beta+\gamma}}{\beta+b}\mu
+\frac{q_{\beta+\gamma}\varphi(\beta)}{\beta+b}+m_{\beta+\gamma})\\
=f_{\beta,\gamma}(\mu,\partial)(\partial+(1+\frac{bq_\gamma}{\beta+b})\mu+\frac{\varphi(\beta)}{b}+\frac{q_\gamma\varphi(\beta)}{\beta+b}+m_{\beta+\gamma}).\nonumber
\end{gather}
If $m\geq 1$, by comparing the coefficients of $\partial^m$ in (\ref{k2}), one can get
\begin{eqnarray*}
a_m(\mu)(\frac{bq_{\beta+\gamma}}{\beta+b}\mu+\frac{q_{\beta+\gamma}\varphi(\beta)}{\beta+b}+m_{\beta+\gamma})+a_m(\mu)m(\frac{b}{\beta+b}\mu+\frac{\varphi(\beta)}{\beta+b})
\\
=a_m(\mu)((1+\frac{bq_\gamma}{\beta+b})\mu+\frac{\varphi(\beta)}{b}+\frac{q_\gamma\varphi(\beta)}{\beta+b}+m_{\beta+\gamma}).
\end{eqnarray*}
Since $a_m(\mu)\neq 0$, we can directly obtain that $q_{\beta+\gamma}-q_\gamma=-m+\frac{\beta+b}{b}$.

It is obvious that when $m=0$, this lemma also holds.
\end{proof}

\begin{lem}\label{lem1}
For any $\beta$, $\gamma\in \Delta $, the degree of $f_{\beta,\gamma}(\lambda,\partial)$ is smaller than $3$.
\end{lem}
\begin{proof}
Letting $\mu=0$ in (\ref{i5}), setting $f_{\beta,\gamma}(0,\partial)=d(\partial)$ and by $\beta+b\notin \Delta$, we get
\begin{eqnarray}
\label{i7}f_{\beta,\gamma}(\lambda,\partial)=\frac{b(d(\lambda+\partial)(\partial+q_{\beta+\gamma}\lambda+m_{\beta+\gamma})-d(\partial)(\partial+q_\gamma\lambda+m_\gamma))}{(\beta+b)\lambda-\varphi(\beta)}.
\end{eqnarray}
Taking (\ref{i7}) into (\ref{i5}) and by some computations, we get
\begin{gather}
((\beta+b)\lambda-b\mu-\varphi(\beta))(\partial+q_{\beta+\gamma}(\lambda+\mu)+m_{\beta+\gamma})
((\beta+b)\mu-\varphi(\beta))d(\lambda+\mu+\partial)\nonumber\\
-((\beta+b)\lambda-b\mu-\varphi(\beta))(\partial+q_\gamma(\lambda+\mu)+m_\gamma)((\beta+b)\mu-\varphi(\beta))d(\partial)\nonumber\\
=b(\partial+q_{\beta+\gamma}\lambda+m_{\beta+\gamma})(\partial+\lambda+q_{\beta+\gamma}\mu+m_{\beta+\gamma})
((\beta+b)(\lambda+\mu)-\varphi(\beta))d(\lambda+\mu+\partial)\nonumber\\
-b(\partial+q_{\beta+\gamma}\lambda+m_{\beta+\gamma})(\partial+\lambda+q_{\gamma}\mu+m_{\gamma})
((\beta+b)(\lambda+\mu)-\varphi(\beta))d(\lambda+\partial)\nonumber\\
-b(\partial+q_{\beta+\gamma}\mu+m_{\beta+\gamma})(\partial+\mu+q_{\gamma}\lambda+m_{\gamma})
((\beta+b)(\lambda+\mu)-\varphi(\beta))d(\mu+\partial)\nonumber\\
\label{i8}+b(\partial+q_{\gamma}\mu+m_{\gamma})(\partial+\mu+q_{\gamma}\lambda+m_{\gamma})
((\beta+b)(\lambda+\mu)-\varphi(\beta))d(\partial).
\end{gather}
Set $d(\lambda)=\sum_{i=0}^md_i\lambda^i$ with $d_m\neq 0$. By Lemma \ref{le3},  $q_{\beta+\gamma}-q_\gamma=-m+\frac{\beta+b}{b}$.
In addition, by considering the terms of $(m+3)$th degree in (\ref{i8}), one can get
\begin{gather}
((\beta+b)\lambda-b\mu)\mu ((\partial+q_{\beta+\gamma}(\lambda+\mu))
(\lambda+\mu+\partial)^m
-(\partial+q_\gamma(\lambda+\mu))\partial^m)\nonumber\\
=b(\partial+q_{\beta+\gamma}\lambda)(\lambda+\mu)((\partial+\lambda+q_{\beta+\gamma}\mu)
(\lambda+\mu+\partial)^m-(\partial+\lambda+q_{\gamma}\mu)(\lambda+\partial)^m)\nonumber\\
\label{ii8}-b(\partial+\mu+q_{\gamma}\lambda)(\lambda+\mu)((\partial+q_{\beta+\gamma}\mu)
(\mu+\partial)^m
-(\partial+q_{\gamma}\mu)\partial^m).
\end{gather}
If $m\geq 3$, by comparing the coefficients of $\lambda^2\partial\mu^m$ in (\ref{ii8}), we can get
\begin{eqnarray}
\label{k5}\frac{\beta+b}{b}(m-1)+\frac{\beta+b}{b}q_{\beta+\gamma}(m^2-m)\\
=\frac{m^2-m+2}{2}+\frac{q_{\beta+\gamma}m}{2}(m^2-2m+5)+q_{\beta+\gamma}^2m(m-1).\nonumber
\end{eqnarray}
Setting $\partial=-q_{\beta+\gamma}\lambda$ in (\ref{ii8}), one can get
\begin{eqnarray}
\label{k6}((\beta+b)\lambda-b\mu)\mu (q_{\beta+\gamma}\mu
((1-q_{\beta+\gamma})\lambda+\mu)^m
-((q_\gamma-q_{\beta+\gamma})\lambda+q_\gamma\mu)(-q_{\beta+\gamma}\lambda)^m)\\
=-b((q_{\gamma}-q_{\beta+\gamma})\lambda+\mu)(\lambda+\mu)(q_{\beta+\gamma}(\mu-\lambda)
(\mu-q_{\beta+\gamma}\lambda)^m
-(q_{\gamma}\mu-q_{\beta+\gamma}\lambda)(-q_{\beta+\gamma}\lambda)^m).\nonumber
\end{eqnarray}
If $m\geq 3$, by comparing the coefficients of $\lambda^2\mu^{m+1}$ and $\lambda^3\mu^m$ in (\ref{k6}), we can get
\begin{eqnarray}
\label{k7}m(\beta+b)q_{\beta+\gamma}-m(\beta+b)q_{\beta+\gamma}^2-bq_{\beta+\gamma}C_m^2+2bq_{\beta+\gamma}^2C_m^2=b(q_\gamma-q_{\beta+\gamma})m
q_{\beta+\gamma}^2+bq_{\beta+\gamma},\\
\label{k8}(\beta+b)C_m^2q_{\beta+\gamma}-2(\beta+b)q_{\beta+\gamma}^2C_m^2+(\beta+b)C_m^2q_{\beta+\gamma}^3-bq_{\beta+\gamma}C_m^3
+3bq_{\beta+\gamma}^2C_m^3-3bC_m^3q_{\beta+\gamma}^3\\
=b(q_\gamma-q_{\beta+\gamma})q_{\beta+\gamma}-bC_m^2(q_\gamma-q_{\beta+\gamma})q_{\beta+\gamma}^3-bmq_{\beta+\gamma}^2 ~~~~(m\geq 4).\nonumber
\end{eqnarray}
Using $q_\gamma-q_{\beta+\gamma}=m-1-\frac{\beta}{b}$, from (\ref{k7}), we can get $q_{\beta+\gamma}=0$ or $q_{\beta+\gamma}=\frac{\beta+b}{b}-\frac{m-1}{2}-\frac{1}{m}$. Suppose $q_{\beta+\gamma}\neq 0$. By (\ref{k8}), one can obtain
\begin{eqnarray}
\label{k9}-bm(m-1)q_{\beta+\gamma}^2+((\beta+b)m(m-1)-\frac{bm}{2}(m^2-3m+4))q_{\beta+\gamma}\\
+\frac{bm}{6}(m^2-3m+8)-\frac{\beta+b}{2}(m^2-m+2)
=0.\nonumber
\end{eqnarray}
Note that $q_{\beta+\gamma}=\frac{\beta+b}{b}-\frac{m-1}{2}-\frac{1}{m}$. Taking it into (\ref{k5}) and (\ref{k9}) , we can get
\begin{eqnarray}
\label{yy1}4\frac{\beta}{b}=m+\frac{2}{m}-5,~~~~m(13-m^2)+\frac{12}{m}=24(1+\frac{\beta}{b}).
\end{eqnarray}
According to the above equalities, we have $m^3-7m-6=0$. Therefore, $m=3$.  By (\ref{yy1}), it follows that $\beta=-\frac{1}{3}b$. Since $\beta\in \Delta$, then $b\in \Delta$, we get a contradiction. Therefore, $q_{\beta+\gamma}=0$ and $q_\gamma=(m-1)-\frac{\beta}{b}$.
By (\ref{k5}), we get
\begin{eqnarray}
\label{yy2}\frac{\beta+b}{b}(m-1)=\frac{m^2-m+2}{2}.
\end{eqnarray}

Letting $\mu=\frac{\beta+b}{b}\lambda$ in (\ref{ii8}), one can get
\begin{gather}
\partial((\partial+\lambda)
(\lambda+\frac{\beta+b}{b}\lambda+\partial)^m-(\partial+\lambda+q_{\gamma}\frac{\beta+b}{b}\lambda)(\lambda+\partial)^m)\nonumber\\
\label{ii9}=(\partial+\frac{\beta+b}{b}\lambda+q_{\gamma}\lambda)(\partial
(\frac{\beta+b}{b}\lambda+\partial)^m
-(\partial+q_{\gamma}\frac{\beta+b}{b}\lambda)\partial^m).
\end{gather}
Comparing the coefficients of $\lambda^4\partial^{m-2}$ in (\ref{ii9}), we obtain
\begin{eqnarray}
\label{ii15}C_m^4(1+\frac{\beta+b}{b})^4+C_m^3(1+\frac{\beta+b}{b})^3-C_m^4-(1+q_\gamma\frac{\beta+b}{b})C_m^3
=C_m^4(\frac{\beta+b}{b})^4+m C_m^3(\frac{\beta+b}{b})^3.
\end{eqnarray}
It can be directly obtained from (\ref{ii15}) that $\frac{\beta}{b}=\frac{3}{4}m-\frac{5}{4}$. Taking it into (\ref{yy2}), we can immediately obtain that $(m-3)(m+1)=0$, which is a contradiction. Therefore, we only need to consider $m=3$ in this case.
Note that in this case, $m=3$, $q_{\beta+\gamma}=0$ and $q_\gamma=2-\frac{\beta}{b}$. Taking it into (\ref{ii8}), dividing
$(\lambda+\mu)\mu$ in both two sides and comparing the coefficients of $\lambda^2\mu\partial$, we get
$2(\beta+b)-b=3b$. Therefore, $\beta=b$, also a contradiction.
Therefore, we get this lemma.
\end{proof}

\begin{lem}\label{lem2}
For any $\beta$, $\gamma\in \Delta$, $f_{\beta,\gamma}(\lambda,\partial)$ must be one of
the following forms:\\
(1) \begin{eqnarray}
\label{h1}f_{\beta,\gamma}(\lambda,\partial)=c_{\beta,\gamma},~~~\text{where $\beta\neq 0$, $q_{\beta+\gamma}-q_\gamma=\frac{\beta}{b}+1$ and $c_{\beta,\gamma}\in \mathbb{C}\setminus\{0\}$};
\end{eqnarray}
(2)
\begin{gather}
\label{h2}f_{\beta,\gamma}(\lambda,\partial)=c_{\beta,\gamma}(m_{\beta+\gamma}+\frac{\varphi(\beta)}{\beta+b}q_{\beta+\gamma}+\frac{b}{\beta+b}q_{\beta+\gamma}\lambda+\partial),\\
\text{where $q_{\beta+\gamma}-q_\gamma=\frac{\beta}{b}$ and $c_{\beta,\gamma}\in \mathbb{C}\setminus\{0\}$};\nonumber
\end{gather}
(3)\begin{gather}
\label{h3}f_{\beta,\gamma}(\lambda,\partial)=c_{\beta,\gamma}(\partial+m_{\beta+\gamma})(\partial+\frac{b}{\beta+b}\lambda+m_{\beta+\gamma}+\frac{\varphi(\beta)}{\beta+b}),\\
\text{where  $\beta\neq 0$, $q_{\beta+\gamma}=0$, $q_\gamma=1-\frac{\beta}{b}$,  and $c_{\beta,\gamma}\in \mathbb{C}\setminus\{0\}$};\nonumber
\end{gather}
(4)\begin{gather}
\label{h4}f_{\beta,\gamma}(\lambda,\partial)=c_{\beta,\gamma}((\partial+m_{\beta+\gamma})(\partial+\frac{2\beta+b}{\beta+b}\lambda+m_{\beta+\gamma}+\frac{(2\beta+b)\varphi(\beta)}{2(\beta+b)})
+\frac{\beta}{\beta+b}(\lambda+\frac{\varphi(\beta)}{b})^2),\\
\text{where $\beta\neq 0$, $q_{\beta+\gamma}=\frac{\beta}{b}$, $q_\gamma=1$,  and $c_{\beta,\gamma}\in \mathbb{C}\setminus\{0\}$}.\nonumber
\end{gather}
\end{lem}
\begin{proof}
By Lemma \ref{lem1}, we only need to discuss the case when the degree of $d(\lambda)$ is smaller than 3 in (\ref{i7}).
Therefore, we can set $f_{\beta,\gamma}(\lambda,\partial)=a_1\lambda^2+a_2\lambda\partial+a_3\partial^2+a_4\lambda+a_5\partial+a_6$, where $a_i\in \mathbb{C}$ for $i=1$, $\cdots$, $6$. Taking it into (\ref{i5}) and by comparing the coefficients of $\lambda^3$, $\lambda^2\mu$, $\lambda^2\partial$, $\lambda^2$, $\lambda\mu$, $\lambda\partial$ and $\lambda$, we can obtain the following equalities:
\begin{eqnarray}
\label{eqss1}a_1(\beta+b)=ba_3q_{\beta+\gamma},\\
\label{eqss2}2(\beta+b)a_1-a_1b=ba_2q_{\beta+\gamma},\\
\label{eqss3}a_2(\beta+b)=ba_3+2ba_3q_{\beta+\gamma},\\
\label{eqss4}a_4(\beta+b)-a_1\varphi(\beta)=ba_3m_{\beta+\gamma}+ba_5q_{\beta+\gamma},\\
\label{eqss5}(\beta+b)a_4-ba_4-2a_1\varphi(\beta)=a_2bm_{\beta+\gamma}+ba_4q_{\beta+\gamma}-ba_4q_\gamma,\\
\label{eqss6}a_5(\beta+b)-\varphi(\beta)a_2=2ba_3m_{\beta+\gamma}+ba_5+ba_5(q_{\beta+\gamma}-q_\gamma),\\
\label{eqss7}a_6(\beta+b)-\varphi(\beta)a_4=ba_5m_{\beta+\gamma}+ba_6(q_{\beta+\gamma}-q_\gamma).
\end{eqnarray}
Note that  the coefficients of other terms in (\ref{i5}) are naturally equal.

Suppose $a_3\neq 0$. By Lemma \ref{le3}, $q_{\beta+\gamma}-q_\gamma=-1+\frac{\beta}{b}$. According to (\ref{eqss1})-(\ref{eqss3}), we get
\begin{eqnarray}
\label{eqss8}\frac{b}{2\beta+b}a_2q_{\beta+\gamma}&=&\frac{b}{\beta+b}q_{\beta+\gamma}a_3,\\
\label{eqss9}a_2 &=& \frac{b(1+2q_{\beta+\gamma})}{\beta+b}a_3.
\end{eqnarray}
By (\ref{eqss8}) and (\ref{eqss9}), we get $\frac{b}{2\beta+b}(1+2q_{\beta+\gamma})q_{\beta+\gamma}=q_{\beta+\gamma}$. Therefore, $q_{\beta+\gamma}=0$ or $q_{\beta+\gamma}=\frac{\beta}{b}$ where $\beta\neq 0$. When $q_{\beta+\gamma}=0$, by (\ref{eqss1}) and (\ref{eqss3})-(\ref{eqss7}), we can directly obtain that $a_1=0$, $a_2=\frac{b}{\beta+b}a_3$,
$a_4=\frac{b}{\beta+b}m_{\beta+\gamma}a_3$, $a_5=(\frac{\varphi(\beta)}{\beta+b}+2m_{\beta+\gamma})a_3$,
and $a_6=\frac{\varphi(\beta)}{2b}a_4+\frac{a_5}{2}m_{\beta+\gamma}$. Therefore, in this case,
$f_{\beta,\gamma}(\lambda,\partial)=a_3(\partial+m_{\beta+\gamma})(\partial+\frac{b}{\beta+b}\lambda+m_{\beta+\gamma}+\frac{\varphi(\beta)}{\beta+b})$. This is Case (3).
Similarly, when $q_{\beta+\gamma}=\frac{\beta}{b}\neq 0$, by some simple computations, we can get Case (4).

When $a_3=0$, by (\ref{eqss1}) and (\ref{eqss2}), we get $a_1=a_2=0$. By (\ref{eqss4})-(\ref{eqss7}),
we can directly get Case (1) and Case (2).

Then this lemma can be obtained.
\end{proof}
\begin{lem}\label{lem3}
Suppose there exist some $\beta_0\neq 0$ and $\gamma_0\in \Delta$ such that $f_{\beta_0,\gamma_0}(\lambda,\partial)$ is of the form (\ref{h3}), then for any $\alpha\neq \beta_0$,
$f_{\alpha,\gamma_0}(\lambda,\partial)$ is of the form (\ref{h2}).
\end{lem}
\begin{proof}
Since $f_{\beta_0,\gamma_0}(\lambda,\partial)$ is of the form (\ref{h3}), we get
$q_{\beta_0+\gamma_0}=0$ and $q_{\gamma_0}=1-\frac{\beta_0}{b}$.

Suppose that there exists some $\beta\neq \beta_0$ such that $f_{\beta,\gamma_0}(\lambda,\partial)$ is of the form (\ref{h3}). Then $q_{\beta+\gamma_0}=0$ and $q_{\gamma_0}=1-\frac{\beta}{b}$. Therefore, $\beta=\beta_0$. We get a contradiction. Therefore, there does not exist any $\beta\neq \beta_0$ such that $f_{\beta,\gamma_0}(\lambda,\partial)$ is of the form (\ref{h3}).

Similarly, for  any $\beta\neq \beta_0$, $f_{\beta,\gamma_0}(\lambda,\partial)$ cannot be of the form (\ref{h4}).

Suppose that there exists some $\alpha$ such that $f_{\alpha,\gamma_0}(\lambda,\partial)$ is of the form (\ref{h1}). Set $\alpha+\beta=\beta_0$ and $\gamma=\gamma_0$ in (\ref{i1}). Therefore we can obtain
$q_{\alpha+\gamma_0}-q_{\gamma_0}=1+\frac{\alpha}{b}$. Then $q_{\alpha+\beta+\gamma}-q_{\alpha+\gamma}
=-2+\frac{\beta}{b}$. Thus, $f_{\beta,\alpha+\gamma}(\lambda,\partial)$ is not one of the forms in Lemma \ref{lem2}. We also get a contradiction.

This lemma follows by the above discussion.
\end{proof}
\begin{lem}\label{lem4}
Suppose there exist some $\beta_0\neq 0$ and $\gamma_0\in \Delta$ such that $f_{\beta_0,\gamma_0}(\lambda,\partial)$ is of the form (\ref{h3}), then for any $\beta\neq 0$, $\gamma\in \Delta$ satisfying $\beta+\gamma=\beta_0+\gamma_0$,
$f_{\beta,\gamma}(\lambda,\partial)$ is of the form (\ref{h3}).
\end{lem}
\begin{proof}
Set $\alpha+\beta=\beta_0$ and $\gamma=\gamma_0$ in (\ref{i1}) with $\beta\neq 0$. Therefore we can obtain
$q_{\alpha+\beta+\gamma}-q_{\gamma}=-1+\frac{\alpha+\beta}{b}$ and $q_{\alpha+\beta+\gamma}=0$. By Lemma \ref{lem3}, $f_{\alpha,\gamma}(\lambda,\partial)$ is of the form (\ref{h2}). Therefore,
$q_{\alpha+\gamma}-q_\gamma=\frac{\alpha}{b}$. Thus, $q_{\alpha+\beta+\gamma}-q_{\alpha+\gamma}=-1+\frac{\beta}{b}$ and $q_{\alpha+\beta+\gamma}=0$. According to Lemma \ref{lem2},
$f_{\beta,\alpha+\gamma}(\lambda,\partial)$ is of the form (\ref{h3}). Since $\beta$ can be any non-zero element in $\Delta$, we get this lemma.
\end{proof}
\begin{lem}\label{lem5}
Suppose there exist some $\beta_0\neq 0$ and $\gamma_0\in \Delta$ such that $f_{\beta_0,\gamma_0}(\lambda,\partial)$ is of the form (\ref{h3}), then for any $\beta$, $\gamma\in \Delta$ satisfying $\gamma\neq \beta_0+\gamma_0$ and $\beta+\gamma\neq \beta_0+\gamma_0$,
$f_{\beta,\gamma}(\lambda,\partial)$ is of the form (\ref{h2}).
\end{lem}
\begin{proof}
It can be directly obtained by Lemma \ref{lem3} and Lemma \ref{lem4}.
\end{proof}
\begin{lem}\label{lem6}
Suppose there exist some $\beta_0\neq 0$ and $\gamma_0\in \Delta$ such that $f_{\beta_0,\gamma_0}(\lambda,\partial)$ is of the form (\ref{h3}),  then for
any $\alpha\neq 0$,  $f_{\alpha,\beta_0+\gamma_0}(\lambda,\partial)$ is of the form (\ref{h1}).
\end{lem}
\begin{proof}
Suppose that there exists some $\beta_0\neq 0$ and $\gamma_0\in \Delta$ such that $f_{\beta_0,\gamma_0}(\lambda,\partial)$ is of the form (\ref{h3}). By Lemma \ref{lem3}, for
any $\alpha\neq 0$, $f_{\alpha+\beta_0,\gamma_0}(\lambda,\partial)$ is of the form (\ref{h2}).
Set $\beta=\beta_0$ and $\gamma=\gamma_0$ in (\ref{i1}).
Since $q_{\alpha+\beta+\gamma}-q_\gamma=\frac{\alpha+\beta}{b}$ and
$q_{\beta+\gamma}=0$, $q_\gamma=1-\frac{\beta}{b}$, we get $q_{\alpha+\beta+\gamma}-q_{\beta+\gamma}
=1+\frac{\alpha}{b}$. Therefore, $f_{\alpha,\beta+\gamma}(\lambda,\partial)$ is of the form (\ref{h1}).
\end{proof}

\begin{notation}
Let $A_1=\{q_\alpha|\alpha\in \Delta\}$ be a sequence  satisfying the following conditions:
\begin{itemize}
\item[(1)] There exists some $\gamma_0\in \Delta$ such that $q_{\gamma_0}=0$;
\item[(2)] $q_{\alpha+\gamma_0}=1+\frac{\alpha}{b} $ with any $\alpha\neq 0$.
\end{itemize}
\end{notation}

\begin{pro}\label{lem7}
Suppose that there exist some $\beta_0\neq 0$ and $\gamma_0\in \Delta$ such that $f_{\beta_0,\gamma_0}(\lambda,\partial)$ is of the form (\ref{h3}). The  module action of
$CL(b,\varphi)$ on $V$ corresponds to the sequence $A_1$ and is given as follows:
\\
${x_{\beta}}_\lambda M_\gamma  =$
\begin{eqnarray*}
 \begin{cases}
c_{\beta,\gamma}(\partial+c-\frac{\varphi(\beta+\gamma)}{b})(\partial+\frac{b}{\beta+b}\lambda+c-\frac{\varphi(\beta+\gamma)}{b} +\frac{\varphi(\beta)}{\beta+b})M_{\beta+\gamma}, &  \text{if $\beta\neq 0$, $q_{\beta+\gamma}=0$, $q_\gamma=1-\frac{\beta}{b}$,}  \\
c_{\beta,\gamma}M_{\beta+\gamma}, &\text{if $\beta\neq 0$, $q_\gamma=0$,\;$q_{\beta+\gamma}=1+\frac{\beta}{b}$},\\
c_{\beta,\gamma}(\partial+\frac{b}{\beta+b}q_{\beta+\gamma}\lambda+\frac{\varphi(\beta)}{\beta+b}q_{\beta+\gamma}
+c-\frac{\varphi(\beta+\gamma)}{b})M_{\beta+\gamma}& \text{otherwise},
\end{cases}
\end{eqnarray*}
where $c\in \mathbb{C}$, and
\begin{eqnarray}
 c_{\alpha+\beta,\gamma}=\frac{\alpha +
\beta+b}{(\alpha+b)(\beta+b)}c_{\beta,\gamma}c_{\alpha,\beta+\gamma}.
\label{coes}
\end{eqnarray}

It is clear that, $c_{\beta,\gamma}=\beta+b$ satisfies (\ref{coes}).
We denote  by $V_{A_1,c}$ the corresponding module.
\end{pro}
\begin{proof}
By Lemmas \ref{lem3}-\ref{lem6} and \ref{le2}, we only need to prove that the action defined above is really a module action, i.e. to check that (\ref{i1}) holds in all cases. This can be checked case by case. Here, we only give a verification when $f_{\alpha+\beta,\gamma}(\lambda,\partial)$ in (\ref{i1}) is of the form (\ref{h1}) with $\alpha\neq 0$ and $\beta\neq 0$. Therefore, $q_\gamma=0$ and
$q_{\alpha+\beta+\gamma}=1+\frac{\alpha+\beta}{b}$. By Lemma \ref{lem6}, $f_{\alpha,\gamma}(\lambda,\partial)$ and $f_{\beta,\gamma}(\lambda,\partial)$ are of the form (\ref{h1}). Therefore, $f_{\alpha,\beta+\gamma}(\lambda,\partial)$ and $f_{\beta,\alpha+\gamma}(\lambda,\partial)$ are of the form (\ref{h2}). Taking them into (\ref{i1}), we get
\[
c_{\alpha+\beta,\gamma}=\frac{bq_{\alpha+\beta+\gamma}}{(\alpha+b)(\beta+b)}c_{\beta,\gamma}c_{\alpha,\beta+\gamma}.
\]
Since $q_{\alpha+\beta+\gamma}=1+\frac{\alpha+\beta}{b}$, we get $c_{\alpha+\beta,\gamma}=\frac{\alpha +
\beta+b}{(\alpha+b)(\beta+b)}c_{\beta,\gamma}c_{\alpha,\beta+\gamma}$.
\end{proof}

\begin{lem}\label{lem8}
Suppose there exist some $\beta_0\neq 0$, $\gamma_0\in \Delta$ such that $f_{\beta_0,\gamma_0}(\lambda,\partial)$ is of the form (\ref{h4}). Then we have\\
(1) $f_{\beta,\gamma_0}(\lambda,\partial)$ is of the form (\ref{h4}), when $\beta\neq 0$;\\
(2) $f_{\beta,\gamma}(\lambda,\partial)$ is of the form (\ref{h2}), when $\beta=0$, $\gamma=\gamma_0$ or $\beta+\gamma\neq \gamma_0$;\\
(3) $f_{\beta,\gamma}(\lambda,\partial)$ is of the form (\ref{h1}) when $\beta\neq 0$ and $\beta+\gamma=\gamma_0$.
\end{lem}
\begin{proof}
The proof is similar to that in the case when there exist some $\beta_0\neq 0$, $\gamma_0\in \Delta$ such that $f_{\beta_0,\gamma_0}(\lambda,\partial)$ is of the form (\ref{h3}).
\end{proof}

\begin{notation}
Let $A_2=\{q_\alpha|\alpha\in \Delta\}$ be a sequence satisfying the following condition:
\begin{itemize}
\item[(1)] There exists some $\gamma_0\in \Delta$ such that $q_{\gamma_0}=0$;
\item[(2)] $q_{\beta+\gamma_0}=\frac{\beta}{b} $ with any $\beta\neq 0$.
\end{itemize}
\end{notation}

\begin{pro}\label{lem13}
Suppose that there exists some $\beta_0\neq 0$ and $\gamma_0\in \Delta$ such that $f_{\beta_0,\gamma_0}(\lambda,\partial)$ is of the form (\ref{h3}). The  module action of
$CL(b,\varphi)$ on $V$ corresponds to the sequence $A_2$ and is given as follows:
\begin{itemize}
\item[(i)] If $\beta\neq 0$, $q_{\beta+\gamma}=1$ and $q_{\gamma}=-\frac{\beta}{b}$, then
${x_{\beta}}_\lambda M_\gamma=c_{\beta,\gamma}M_{\beta+\gamma}$;
\item[(ii)] If $\beta\neq 0$, $q_{\beta+\gamma}=\frac{\beta}{b}$, $q_\gamma=1$, then
${x_{\beta}}_\lambda M_\gamma=c_{\beta,\gamma}((\partial+c-\frac{\varphi(\beta+\gamma)}{b})(\partial+\frac{2\beta+b}{\beta+b}\lambda+c-\frac{\varphi(\beta+\gamma)}{b} +\frac{(2\beta+b)\varphi(\beta)}{2(\beta+b)})+\frac{\beta}{\beta+b}(\lambda+\frac{\varphi(\beta)}{b})^2)M_{\beta+\gamma}$;
\item[(iii)] Otherwise, then ${x_{\beta}}_\lambda M_\gamma=c_{\beta,\gamma}(\partial+\frac{b}{\beta+b}q_{\beta+\gamma}\lambda+\frac{\varphi(\beta)}{\beta+b}q_{\beta+\gamma}
+c-\frac{\varphi(\beta+\gamma)}{b})M_{\beta+\gamma}$.
\end{itemize}
where $c\in \mathbb{C}$, and $c_{\beta,\gamma}=\beta+b$ satisfies (\ref{coes}).
We also denote the corresponding module by $V_{A_2,c}$.
\end{pro}
\begin{proof}
The proof is similar to that in Proposition \ref{lem7}.
\end{proof}

\begin{notation}
Let $A_3=\{q_\alpha|\alpha\in \Delta\}$ be a sequence satisfying
$q_{\beta+\gamma}-q_\gamma=\frac{\beta}{b}$ for any $\beta$, $\gamma\in \Delta$.
Setting $q_0=e$ for some $e\in \mathbb{C}$, then $q_\beta=e+\frac{\beta}{b}$.
\end{notation}

\begin{pro}\label{lem9}
Suppose that there does not exist some $\beta_0$, $\gamma_0\in \Delta$ such that $f_{\beta_0,\gamma_0}(\lambda,\partial)$ is of the form (\ref{h3}) or (\ref{h4}). Then the module action of
$CL(b,\varphi)$ on $V$ corresponds to $A_3$ and is given as follows:
\begin{eqnarray*}
{x_{\beta}}_\lambda M_\gamma=c_{\beta,\gamma}(c-\frac{\varphi(\beta+\gamma)}{b}+\frac{\varphi(\beta)}{\beta+b}(e+\frac{\beta+\gamma}{b})
+\frac{b}{\beta+b}(e+\frac{\beta+\gamma}{b})\lambda+\partial)M_{\beta+\gamma},
\end{eqnarray*}
where $c\in \mathbb{C}$, $e\in \mathbb{C}$ and $c_{\beta,\gamma}$ satisfies (\ref{coes}). We denote the corresponding module by
$V_{c,e}$ with $c_{\beta,\gamma}=\beta+b$.
\end{pro}
\begin{proof}
By Lemma \ref{lem7} and the assumption, for any fixed $\beta$, $\gamma\in \Delta$,
$f_{\beta,\gamma}(\lambda,\partial)$ is either of the form (\ref{h1}) or of the form (\ref{h2}).

Obviously, if all $f_{\beta,\gamma}(\lambda,\partial)$ are of the form (\ref{h1}), (\ref{i1}) cannot hold. Therefore, there exist some $\beta_0\neq 0$ and $\gamma_0\in \Delta$ such that $f_{\beta_0,\gamma_0}(\lambda,\partial)$ is of the form (\ref{h2}). Set $\alpha+\beta=\beta_0$ and
$\gamma=\gamma_0$ in (\ref{i1}). Note that $q_{\alpha+\beta+\gamma}-q_{\gamma}=\frac{\alpha+\beta}{b}$.
In (\ref{i1}), if $f_{\beta,\gamma}(\lambda,\partial)$ is of the form as in
(\ref{h1}), then $q_{\beta+\gamma}-q_\gamma=1+\frac{\beta}{b}$. Therefore, $q_{\alpha+\beta+\gamma}-q_\gamma
=\frac{\alpha}{b}-1$. By our assumption,
we get a contradiction. Therefore, $f_{\beta,\gamma}(\lambda,\partial)$ and $f_{\alpha,\beta+\gamma}(\lambda,\partial)$
must be of the form as in (\ref{h2})). With a similar discussion, $f_{\alpha,\gamma}(\lambda,\partial)$ and $f_{\beta,\alpha+\gamma}(\lambda,\partial)$ are also
of the form as in (\ref{h2}). Note that here $\alpha$ can be any element in $\Delta$. Therefore, in this case, all $f_{\beta,\gamma}(\lambda,\partial)$ are of the form (\ref{h2}).
Taking them into (\ref{i1}), we can directly obtain that (\ref{i1}) holds if and only if (\ref{coes}) is satisfied.
\end{proof}

\begin{lem}\label{lem10}
For $(\ref{coes})$, there exist some non-zero numbers $e_\alpha\in \mathbb{C}$ such that
\[
c_{\alpha,\beta}=\frac{(\alpha+b)(\beta+b)}{\alpha+\beta+b}\frac{e_{\alpha+\beta}}{e_{\beta}}.
\]
\end{lem}
\begin{proof}
Since $V$ is a non-trivial $CL(b,\varphi)$-module, all $c_{\alpha,\beta}$ are not equal to $0$. Setting $\gamma=0$ in (\ref{f1}), we can get that $c_{\alpha,\beta}=\frac{(\alpha+b)(\beta+b)}{\alpha+\beta+b}\frac{c_{\alpha+\beta,0}}{c_{\beta,0}}$.
Let $e_{\alpha}=c_{\alpha,0}$ for any $\alpha\in\Delta$. Then we get that $c_{\alpha,\beta}=\frac{(\alpha+b)(\beta+b)}{\alpha+\beta+b}\frac{e_{\alpha+\beta}}{e_{\beta}}$.
\end{proof}

\begin{thm}\label{th3}
Assume that $V$ is a non-trivial $\Delta$-graded free intermediate module over $CL(b,\varphi)$. Then $V$ is isomorphic to either
$V_{c,e}$ or $V_{A_1,c}$ or $V_{A_2,c}$.
\end{thm}
\begin{proof}
Set $M_\gamma^{'}=\frac{e_\gamma}{\gamma+b}M_\gamma$ for any $\gamma\in \Delta$. Then this theorem can be directly obtained from Propositions \ref{lem7}, \ref{lem9}, \ref{lem13} and Lemma \ref{lem10} by
replacing $M_\gamma$ by $M_\gamma^{'}$.
\end{proof}
\begin{rmk}
According to the relation between the conformal modules of Lie conformal algebra and the modules of its coefficient algebra, by Theorem \ref{th3}, we can naturally obtain some representations of infinite-dimensional Lie algebra $\text{Coeff}(CL(b,\varphi))$.
\end{rmk}

{\bf Acknowledgments}
{This work was supported by the Scientific Research Foundation of Hangzhou Normal University (No. 2019QDL012) and the National Natural Science Foundation of China (No. 11871421, 11501515).}

\end{document}